\documentclass[11pt,reqno]{article} 

\usepackage{graphicx}
\graphicspath{ {./images/} }
\usepackage{lipsum}
\usepackage{subcaption}

\usepackage{xcolor}

\usepackage{amsthm}
\usepackage{bbm}
\usepackage{mathrsfs}
\usepackage{dynkin-diagrams}

\usepackage[utf8]{inputenc}

\usepackage{tikz-cd}

\usepackage{enumitem}

\usepackage{amssymb}
\usepackage{float}
\usepackage{amsbsy}
\usepackage[all]{xy}\usepackage{xypic}
\usepackage{braket}
\usepackage[colorlinks = true,citecolor=black]{hyperref}
\usepackage{amsmath}

\usepackage{wrapfig}

\usepackage{authblk}

\usepackage[upint]{stix}

\newcommand\myeq{\mathrel{\stackrel{\makebox[0pt]{\mbox{\normalfont\tiny loc}}}{=}}}

\hypersetup{
     colorlinks=true,
     linkcolor=blue,
     filecolor=blue,
     citecolor = blue,      
     urlcolor=cyan,
     }

\makeatletter
\newcommand\opteq[1]{\mathrel{\mathpalette\opt@eq{#1}}}
\newcommand{\opt@eq}[2]{%
  \begingroup
  \sbox\z@{$#1#2$}%
  \sbox\tw@{\resizebox{!}{.5\ht\z@}{$\m@th#1($}}%
  \nonscript\hskip-\wd\tw@
  \mkern1mu
  \raisebox{-.35\ht\z@}[0pt][0pt]{\resizebox{!}{.5\ht\z@}{$\m@th#1($}}%
  \mkern-1mu
  {#2}%
  \mkern-1mu
  \raisebox{-.35\ht\z@}[0pt][0pt]{\resizebox{!}{.5\ht\z@}{$\m@th#1)$}}%
  \mkern1mu
  \nonscript\hskip-\wd\tw@
  \endgroup
}
\makeatother

\def\XXint#1#2#3{{\setbox0=\hbox{$#1{#2#3}{\int}$}
     \vcenter{\hbox{$#2#3$}}\kern-.5\wd0}}

\newtheorem{theorem}{Theorem}[section]
\newtheorem*{theorem*}{Theorem}
\newtheorem{theorem-non}{Theorem}
\newtheorem{lemma-non}{Lemma}

\theoremstyle{definition} 
\newtheorem{thm}{Theorem}

\theoremstyle{definition} 
\newtheorem{corollarynon}{Corollary}

\newtheorem{conjecture-non}{Conjecture}

\newtheorem{corollary-non}{Corollary}
\newtheorem{proposition}[theorem]{Proposition}
\newtheorem{lemma}[theorem]{Lemma}
\newtheorem*{lemma*}{Lemma}

\newtheorem*{conjecture*}{Conjecture}

\theoremstyle{definition}
\newtheorem{definition}[theorem]{Definition}

\theoremstyle{remark}
\newtheorem{remark}[theorem]{Remark}

\numberwithin{equation}{section}

\usepackage{geometry}
\begin{document}
\title{{\bf{The dHYM equation on crepant resolutions of Calabi-Yau cones}}}

\author[1]{Eder M. Correa}

\affil[1]{Universidade Estadual de Campinas, Brazil\\

Instituto de Matemática, Estatística e Computação Científica}
\affil[1]{ederc@unicamp.br}


\maketitle

\begin{abstract}
In this paper, we study the deformed Hermitian-Yang-Mills (dHYM) equation on the non‑compact Calabi-Yau manifold $Z = {\rm{Tot}}({\bf{K}}_{X})$, where $X$ is a rational homogeneous variety. Since $X$ is a Fano variety, $Z$ is a resolution of the singularity at the vertex of the affine cone ${\rm{Aff}}(X)$, provided the cone is built using the anticanonical polarization ${\bf{L}} = {\bf{K}}_{X}^{-1}$. Using the cohomogeneity-one symmetry of the Ricci-flat K\"{a}hler metric obtained via the Calabi ansatz on $Z$, we reduce the fully nonlinear PDE underlying the dHYM equation to a scalar, asymptotically autonomous ordinary differential equation (ODE). From this, we determine the exact condition on the topological phase that guarantees global existence of solutions. As an application, we show that every holomorphic line bundle over $Z$ admits a smooth, globally defined Hermitian connection solving the dHYM equation, provided the total phase lies in an explicit open interval determined by the Lie-theoretic data. Also, we prove a rigidity result classifying the exact geometric conditions under which the dHYM solution collapses into a classical Hermitian-Yang-Mills (HYM) connection. The results established generalize previous constructions and provide a substantial new class of examples. Furthermore, the approach presented allows one to study the behavior of the dHYM solutions through ODE methods. Using this approach, we construct the first explicit non-trivial example of a Hermitian-Einstein connection on a line bundle over a non-toric Calabi-Yau manifold which is not dHYM.
\end{abstract}

\hypersetup{linkcolor=blue}
\tableofcontents

\hypersetup{linkcolor=black}

\maketitle

\newpage

\section{Introduction}

Let $(X,\omega)$ be a compact connected K\"{a}hler manifold of complex dimension $n$. The deformed Hermitian-Yang-Mills (dHYM) equation was originally discovered in the context of string theory and mirror symmetry, describing the BPS states of D-branes \cite{marino2000nonlinear, leung2000special}. Mathematically, the problem asks for a smooth $(1, 1)$-form $\chi$ in a given real cohomology class $[\psi] \in H^{1,1}(X, \mathbbm{R})$ such that 
\begin{equation}
\label{dHYMintro}
    {\textrm{Im}}\big(\omega + \sqrt{-1}\chi\big)^{n} = \tan(\hat{\Theta}) {\textrm{Re}}\big(\omega + \sqrt{-1}\chi\big)^{n},
\end{equation}
where $\hat{\Theta}$ is a topological constant. According to the Strominger-Yau-Zaslow (SYZ) conjecture, a mirror pair of Calabi-Yau threefolds can be fibered by special Lagrangian tori in certain limits. Under this fibration, a real Fourier-Mukai transform identifies solutions of the dHYM equation with special Lagrangian sections on the mirror side. This establishes a deep link between calibrated geometry and gauge theory for mirror Calabi-Yau threefolds. For a detailed discussion on the physical origins of the dHYM equation and its relation to string theory, we refer to \cite{Collins2018deformed}.

As demonstrated by Jacob and Yau \cite{JacobYau2017}, the highly nonlinear partial differential equation (\ref{dHYMintro}) can be equivalently formulated through the notion of Lagrangian phase. Significant progress has been made regarding the solvability of the dHYM equation on compact K\"{a}hler manifolds. For general projective varieties, existence results typically rely on curvature constraints or topological bounds, such as the supercritical phase condition \cite{JacobYau2017, Pingali2019, Collins2020, takahashi2020tan}. These conditions are deeply motivated by algebraic concepts like Bridgeland stability and the Collins-Jacob-Yau conjecture, which has been extensively studied and proved in several cases \cite{Chen2021j, ballal2023supercritical, chu2024nakai}. Recently, in \cite{correa2024dhym}, the author constructs the first explicit non-trivial example of dHYM connection on a higher rank slope-unstable holomorphic vector bundle over a Fano threefold, and proved in \cite{correa2026deformed} that the dHYM equation is unobstructed on rational homogeneous varieties. More recently, Charbonneau, Oliveira, and Sena-Dias \cite{charbonneau2026deformed} constructed the first irreducible higher-rank dHYM connections on the full flag variety of ${\rm{SL}}_{3}(\mathbbm{C})$ in the small radius regime.

In parallel, a few explicit solutions on non‑compact manifolds have been obtained. Fowdar \cite{fowdar2024examples, fowdar2024explicit} constructed one‑parameter families of abelian dHYM connections on the local Calabi–Yau manifolds $\mathscr{O}_{\mathbbm{P}^{1}}(-2), \mathscr{O}_{\mathbbm{P}^{2}}(-3)$, and $T^{\ast}\mathbbm{P}^{2}$. To date, besides these isolated examples, no systematic construction of dHYM solutions on non‑compact Calabi–Yau manifolds has been available.

In this paper, we investigate the dHYM equation on local Calabi-Yau manifolds, specifically the canonical bundle $Z = {\rm{Tot}}({\bf{K}}_{X})$ over a rational homogeneous variety $X = G^{\mathbbm{C}}/P$. Due to the triviality of its first Chern class, the non-compact total space $Z$ admits a complete Ricci-flat K\"{a}hler metric $\omega_{{\rm{CY}}}$ constructed via the celebrated Calabi ansatz \cite{Calabi1979}. By exploiting the cohomogeneity-one symmetry of this geometry, we bridge the global algebraic structure of rational homogeneous varieties with the local analytical mechanics of the Calabi-Yau metric.

Our main result, Theorem \ref{TheoremA}, shows that the fully nonlinear Monge-Ampère type PDE governing the dHYM equation on $Z$ can be explicitly reduced to an asymptotically autonomous ordinary differential equation (ODE). This exact reduction allows us to bypass standard elliptic estimates, providing a rigorous classification of the solutions, their phase boundaries, and their asymptotic behavior at infinity. Building on this exact reduction, in Corollary \ref{CorollaryA} we prove that every holomorphic line bundle over the non-compact Calabi-Yau manifold $Z = {\rm{Tot}}({\bf{K}}_{X})$ admits a smooth, globally defined Hermitian connection solving the dHYM equation, for every rational homogeneous variety $X$. Moreover, we prove a rigidity result classifying the exact geometric conditions under which the dHYM solution collapses into a classical Hermitian-Yang-Mills (HYM) connection. Applying this framework, we construct the first explicit non-trivial example of a Hermitian-Einstein connection on a line bundle over a non-toric Calabi-Yau manifold that is not dHYM. The results established generalize in a systematic way previous constructions and provide a substantial new class of examples of solutions to the dHYM on non-compact Calabi-Yau manifolds.

\newpage

\subsection*{Main Results}

In order to state our main results, let us introduce the basic set-up. Let $X = G^{\mathbbm{C}}/P$ be a rational homogeneous variety (a flag variety), where $G^{\mathbbm{C}}$ is a complex simple Lie group and $P \subset G^{\mathbbm{C}}$ is a parabolic subgroup. We denote by
\begin{equation}
p \colon Z = {\rm{Tot}}({\bf{K}}_{X}) \to X
\end{equation}
the total space of the canonical bundle of $X$. Since $X$ is Fano, $Z$ is a non-compact Calabi--Yau manifold, and we equip it with the complete Ricci-flat Kähler metric $\omega_{{\rm{CY}}}$ obtained from the Calabi ansatz (e.g. \cite{Calabi1979} and \cite{HwangSinger2002}). Concretely, if $s$ denotes the squared norm on the fibers of ${\bf{K}}_{X}$ with respect to the Hermitian metric induced by the $G$-invariant Kähler-Einstein metric $\omega_{0}$ on $X$ (normalized so that ${\rm{Ric}}(\omega_{0}) = \omega_{0}$), then
\begin{equation}
\omega_{{\rm{CY}}} = U(s)\, p^{\ast}\omega_{0} + V(s)\, \frac{\sqrt{-1}\,\partial s \wedge \bar{\partial} s}{s},
\end{equation}
where $V(s)= U'(s)$, together with $U(s) > 0$ and $V(s) > 0$ for all $s \geq 0$. For any real cohomology class $[\eta] \in H^{1,1}(Z,\mathbbm{R})$, since 
\begin{equation}
p^{\ast} \colon H^{1,1}(X,\mathbbm{R}) \to H^{1,1}(Z,\mathbbm{R})
\end{equation}
is an isomorphism, there exists a unique class $[\chi] \in H^{1,1}(X,\mathbbm{R})$ such that $[\eta] = p^{\ast}[\chi]$. Following \cite{correa2026deformed}, the endomorphism $\omega_{0}^{-1} \circ \chi$ has constant eigenvalues ${\bf{q}}_{\beta}(\omega_{0}^{-1} \circ \chi)$, indexed by the positive roots $\beta \in \Phi_{I}^{+}$, and given by the Lie-theoretic data of $X$ via
\begin{equation}
{\bf{q}}_{\beta}(\omega_{0}^{-1} \circ \chi) = \frac{\langle \lambda([\chi]), \beta^{\vee} \rangle}{\langle \lambda([\omega_{0}]), \beta^{\vee} \rangle}, \ \ \ \beta \in \Phi_{I}^{+}.
\end{equation}
We denote by
\begin{equation}
\Theta_{\omega_{0}}(\chi) = \sum_{\beta \in \Phi_{I}^{+}} \arctan\!\big( {\bf{q}}_{\beta}(\omega_{0}^{-1} \circ \chi) \big)
\end{equation}
the Lagrangian phase of $\chi$ with respect to $\omega_{0}$, and by $\hat{\Theta}_{{\rm{tot}}}$ the total topological phase of the dHYM equation on $Z$. With this notation, the main result of this work can be stated as follows.

\begin{thm}
\label{TheoremA}
In the above setting, given $[\eta] \in H^{1,1}(Z, \mathbbm{R})$, let $\chi \in \Omega^{1,1}(X)$ be the unique basic $G$-invariant closed $(1,1)$-form such that $[\eta] = p^{\ast}[\chi]$. For any real-valued smooth function $\phi \in C^{\infty}([0,+\infty))$, define the ansatz 
\begin{equation}
\Upsilon_{\phi} = p^{\ast}\chi - \sqrt{-1}\partial \bar{\partial} \phi(s) \in [\eta],
\end{equation}
where $s$ denotes the squared norm on the fibers. Then the following hold:
\begin{enumerate}
    \item[(1)] The dHYM equation 
    \begin{equation}
    {\rm{Im}}\big (\omega_{{\rm{CY}}} + \sqrt{-1}\Upsilon_{\phi} \big )^{n+1} = \tan(\hat{\Theta}_{{\rm{tot}}}) {\rm{Re}}\big (\omega_{{\rm{CY}}} + \sqrt{-1}\Upsilon_{\phi} \big )^{n+1}
    \end{equation}
    globally reduces to the following ordinary differential equation 
    \begin{equation}
    \label{dHYMODE}
    \frac{d}{ds}\psi(s) = V(s) \tan\Bigg ( \displaystyle \hat{\Theta}_{{\rm{tot}}} - \sum_{\beta \in \Phi_{I}^{+}} \arctan\Bigg( \frac{{\bf{q}}_{\beta}(\omega_{0}^{-1} \circ \chi) + \psi(s)}{U(s)} \Bigg) \Bigg),
    \end{equation}
    where ${\bf{q}}_{\beta}(\omega_{0}^{-1} \circ \chi)$ are the constant eigenvalues of $\omega_{0}^{-1} \circ \chi$, and $U(s), V(s)$ are the geometric scale factors defining $\omega_{{\rm{CY}}}$, with the initial condition $\psi(0)=0$.
    
    \item[(2)] If the phase calibration $\hat{\Theta}_{{\rm{tot}}}$ satisfies 
    \begin{equation}
    \hat{\Theta}_{{\rm{tot}}} \in \big (\Theta_{\omega_{0}}(\chi) -\tfrac{\pi}{2},\; \Theta_{\omega_{0}}(\chi) + \tfrac{\pi}{2}\big), 
    \end{equation}
    then the initial value problem of item (1) admits a unique global solution $\psi(s)$ for all $s \ge 0$. Consequently $\Upsilon_{\phi} \in [\eta]$ such that 
\begin{equation}
\phi(s) = -\int_{0}^{s}\frac{\psi(u)}{u}du, \ \ s \in [0,+\infty),
\end{equation}
defines a smooth solution to the dHYM equation on $Z$. 
\end{enumerate}
\end{thm}
Recently, the Calabi ansatz and symmetry reduction have emerged as powerful tools for studying the dHYM equation on non-compact Kähler manifolds, see for instance \cite{sheu2021deformed}, \cite{jacob2022deformed}, \cite{fowdar2024examples, fowdar2024explicit}. The result established in Theorem \ref{TheoremA} provides a new ODE-based method for studying the asymptotic behavior of dHYM solutions. The above result generalizes previous constructions and yields a substantial new class of examples. As a particular instance, we can solve the dHYM equation on the unique toric crepant resolution provided by the Cartan-Remmert reduction
\begin{equation}
\mathscr{R} \colon {\rm{Tot}}(\mathscr{O}_{\mathbbm{P}^{n-1}}(-n)) \to \mathbbm{C}^{n}/\mathbbm{Z}_{n},
\end{equation}
for all $n > 1$. To the best of our knowledge, the above theorem establishes the first systematic construction and robust global existence theorem for dHYM solutions on non-compact Calabi-Yau manifolds, extending significantly beyond previously known isolated examples.

Applying the framework of Theorem \ref{TheoremA} to holomorphic line bundles, we prove that every line bundle over $Z = {\rm{Tot}}({\bf{K}}_{X})$ admits a smooth, globally defined Hermitian connection solving the dHYM equation, provided the total phase lies in an explicit open interval determined by the Lie-theoretic data of $X$. More precisely, we have the following corollary.

\begin{corollarynon}
\label{CorollaryA}
Let $\mathcal{L} = p^{\ast}{\bf{L}} \to (Z,\omega_{{\rm{CY}}})$ be the pullback of a holomorphic line bundle ${\bf{L}} \to X_{P}$ over the non-compact Calabi-Yau manifold $Z = {\rm{Tot}}({\bf{K}}_{X_{P}})$. Then $\mathcal{L}$ admits a smooth, globally defined, Hermitian metric $h$ solving the deformed Hermitian Yang-Mills equation
\begin{equation}
    {\rm{Im}}\big (\omega_{{\rm{CY}}} - F_{h} \big )^{n+1} = \tan(\hat{\Theta}_{{\rm{tot}}}) {\rm{Re}}\big (\omega_{{\rm{CY}}} -F_{h} \big )^{n+1},
\end{equation}
for any choice of the total phase calibration strictly within the interval
\begin{equation}
    \hat{\Theta}_{{\rm{tot}}} \in \big ( \hat{\Theta}({\bf{L}}) - \frac{\pi}{2}, \hat{\Theta}({\bf{L}}) + \frac{\pi}{2} \big),
\end{equation}
where $\hat{\Theta}({\bf{L}}) = \displaystyle {\textrm{Arg}} \int_{X_{P}}\frac{(\omega_{0} + \sqrt{-1}\chi)^{n}}{n!}$, and $\chi \in  c_{1}({\bf{L}})$ is the unique $G$-invariant representative. In particular, if both of the following conditions hold:
\begin{enumerate}
\item[(i)] ${\bf{q}}_{\beta}(\omega_{0}^{-1} \circ \chi) = q$, $\forall \beta \in \Phi_{I}^{+}$,
\item[(ii)] $\hat{\Theta}_{{\rm{tot}}} = (n+1)\arctan(q)$,
\end{enumerate}
then $\mathcal{L}$ admits a smooth, globally defined, Hermitian metric $h$ solving the system of equations
\begin{equation}
\begin{cases} 
 \displaystyle {\rm{Im}}\big (\omega_{{\rm{CY}}} - F_{h} \big )^{n+1} = \tan(\hat{\Theta}_{{\rm{tot}}}) {\rm{Re}}\big (\omega_{{\rm{CY}}} -F_{h} \big )^{n+1}, \\
 \\
\displaystyle \sqrt{-1}\Lambda_{\omega_{{\rm{CY}}}}(F_{h}) = (n+1)q,
\end{cases}
\end{equation}
i.e., the induced Chern connection $\nabla^{h}$ is both Yang-Mills and deformed Hermitian Yang-Mills.
\end{corollarynon}

In the setting of Corollary \ref{CorollaryA}, since every holomorphic line bundle $\mathcal{L} \to Z$ is isomorphic to $p^{\ast}{\bf{L}}$, for some (unique) holomorphic line bundle ${\bf{L}} \to X_{P}$, the above result shows that, for every $\mathcal{L} \in {\rm{Pic}}(Z)$, the dHYM equation is unobstructed for every total phase
\begin{center}
$\hat{\Theta}_{{\rm{tot}}} \in \big ( \hat{\Theta}({\bf{L}}) - \frac{\pi}{2}, \hat{\Theta}({\bf{L}}) + \frac{\pi}{2} \big)$. 
\end{center}
The constructive nature of the above results allows us to explore the dHYM equation on the non-compact Calabi-Yau manifold $Z = {\rm{Tot}}({\bf{K}}_{X_{B}})$, where $X_{B} = {\rm{SL}}_{3}(\mathbbm{C})/B$ is the non-toric rational homogeneous variety (full flag variety) of ${\rm{SL}}_{3}(\mathbbm{C})$.  In this setting, we consider the pullback line bundle $\mathcal{L} = p^{\ast}{\bf{L}}$ with ${\bf{L}} = \mathscr{O}_{\alpha_{1}}(a) \otimes \mathscr{O}_{\alpha_{2}}(a)$, $a > 0$. For the HYM equation $\Lambda_{\omega_{{\rm{CY}}}}(\sqrt{-1}F_{h}) = C$ on $\mathcal{L}$, our reduction provides the family of explicit global solutions
\begin{equation}
\psi_{C}(s) = \frac{C}{4}U(s) - q + \Big( q - \frac{C}{4} \Big)\frac{1}{U(s)^{3}}, \qquad q = \frac{a}{4\pi},
\end{equation}
where $U(s) = (1+cs)^{1/4}$ is the base scale factor of the Calabi--Yau metric. If $C = 4q$, the solution also solves the dHYM equation with phase $\hat{\Theta}_{{\rm{tot}}} = 4\arctan(q)$, in accordance with Corollary \ref{CorollaryA}. However, for any $C \neq 4q$, the Lagrangian phase
\begin{equation}
\Theta(s) = 3\arctan\!\big( x_{C}(s) \big) + \arctan\!\big( y_{C}(s) \big), \qquad x_{C}(s) = \frac{q + \psi_{C}(s)}{U(s)}, \quad y_{C}(s) = \frac{\psi_{C}'(s)}{V(s)},
\end{equation}
is not constant on the radial coordinate $s$, so the resulting Hermitian-Einstein connection solves the HYM equation but fails to solve the dHYM equation. This yields the first explicit non-trivial example of a Hermitian-Einstein connection that is not dHYM, on a line bundle over a non-toric Calabi-Yau manifold. The example illustrates the finer nature of the dHYM equation compared to the classical HYM equation and demonstrates the effectiveness of the ODE reduction in producing computable solutions.

\section{Generalities on Semisimple Lie Theory}

In this section, we introduce the fundamental concepts and notation of semisimple Lie groups and algebras, establishing the algebraic framework required for the geometric constructions discussed throughout the paper. For more details, we suggest \cite{HumphreysLAG} and \cite{Humphreys}.

\subsubsection*{Highest Weight Modules}

From now on, we fix a complex semisimple Lie algebra $\mathfrak{g}^{\mathbbm{C}}$, and a triangular decomposition
\begin{equation}
\label{triangular}
\mathfrak{g}^{\mathbbm{C}} = \mathfrak{n}^{-} \oplus \mathfrak{h} \oplus \mathfrak{n}^{+}, 
\end{equation}
induced by a triple $(\mathfrak{g}^{\mathbbm{C}},\mathfrak{h},\Delta)$, where $\mathfrak{h}$ is a Cartan subalgebra, 
\begin{equation}
\mathfrak{n}^{-} = \sum_{\alpha \in \Phi^{-}}\mathfrak{g}_{\alpha} \  \ \ \ \text{and} \ \ \ \  \mathfrak{n}^{+} = \sum_{\alpha \in \Phi^{+}}\mathfrak{g}_{\alpha}.
\end{equation}
Here we denote by $\Phi = \Phi^{+} \cup \Phi^{-}$ the root system associated with the simple root system $\Delta \subset \mathfrak{h}^{\ast}$. Furthermore, we shall denote by $G^{\mathbbm{C}}$ the connected, simply connected, and complex Lie group such that ${\rm{Lie}}(G^{\mathbbm{C}}) = \mathfrak{g}^{\mathbbm{C}}$. In this setting, we consider the following definitions.

\begin{definition}
A Lie subalgebra $\mathfrak{b} \subset \mathfrak{g}^{\mathbbm{C}}$ is called a Borel subalgebra if $\mathfrak{b}$ is a maximal solvable subalgebra of $\mathfrak{g}^{\mathbbm{C}}$. 
\end{definition}
\begin{definition}
A Lie subgroup $B \subset G^{\mathbbm{C}}$ is called a Borel subgroup if ${\rm{Lie}}(B) \subset \mathfrak{g}^{\mathbbm{C}}$ is a Borel subalgebra.
\end{definition}
Associated to the decomposition given in Eq. (\ref{triangular}) we can define a Borel subalgebra by setting $\mathfrak{b}(\Delta) := \mathfrak{h} \oplus \mathfrak{n}^{+}$. In this last case, $\mathfrak{b}(\Delta)$ is called standard Borel subalgebra relative to $\mathfrak{h}$. Now we have the following result (see for instance \cite{Flagvarieties}, \cite{HumphreysLAG}, \cite{Humphreys}):
\begin{theorem}
\label{Borelconjugate}
Any two Borel subalgebras (subgroups) are conjugate.
\end{theorem}
From the result above, given a Borel subgroup $B \subset G^{\mathbbm{C}}$, up to conjugation, we can always suppose that $B = \exp(\mathfrak{b})$, with $\mathfrak{b} = \mathfrak{b}(\Delta)$.

 Let us recall some basic facts about the representation theory of complex semisimple Lie algebras, a detailed exposition of which can be found in \cite{Humphreys}. By keeping the previous notation, for every $\alpha \in \Phi$, we set 
\begin{equation}
\alpha^{\vee} := \frac{2}{\langle \alpha, \alpha \rangle}\alpha.
\end{equation}
\begin{remark}
In what follows, $\forall \phi,\psi \in \mathfrak{h}^{\ast}$, we denote 
\begin{equation}
\langle \phi, \psi \rangle = \kappa (t_{\phi},t_{\psi}), 
\end{equation}
where $t_{\phi},t_{\psi} \in  \mathfrak{h}$ are, respectively, the dual of $\phi$ and $\psi$ with respect to the Killing form $\kappa$.
\end{remark}
In this setting, the fundamental weights $\{\varpi_{\alpha} \ | \ \alpha \in \Delta\} \subset \mathfrak{h}^{\ast}$ of $(\mathfrak{g}^{\mathbbm{C}},\mathfrak{h})$ are defined by requiring that $\langle \varpi_{\alpha}, \beta^{\vee} \rangle= \delta_{\alpha \beta}$, $\forall \alpha, \beta \in \Delta$. We denote by 
\begin{equation}
\Lambda = \bigoplus_{\alpha \in \Delta}\mathbbm{Z}\varpi_{\alpha}, \ \ \ \Lambda^{+} = \bigoplus_{\alpha \in \Delta}\mathbbm{Z}_{\geq 0}\varpi_{\alpha}, \ \ \ \Lambda^{++} = \bigoplus_{\alpha \in \Delta}\mathbbm{Z}_{> 0}\varpi_{\alpha},
\end{equation}
respectively, the set of weights, the set of integral dominant weights, and the set of strongly dominant weights of $\mathfrak{g}^{\mathbbm{C}}$. From above, given $\alpha \in \Delta$, it follows that 
\begin{equation}
\label{Cartanchange}
\alpha = \sum_{\beta \in \Delta} C_{\alpha \beta}\varpi_{\beta}, \ \ {\text{s.t.}} \ \ C_{\alpha \beta} = \langle \alpha,\beta^{\vee} \rangle \in \mathbbm{Z}, \ \forall \alpha, \beta \in \Delta.
\end{equation}
In other words, the Cartan matrix $C = (C_{\alpha \beta})$ of $\mathfrak{g}^{\mathbbm{C}}$ expresses the change of basis defined by simple roots and fundamental weights. Let $\varrho \colon \mathfrak{g}^{\mathbbm{C}} \to \mathfrak{gl}(V)$ be an arbitrary finite-dimensional $\mathfrak{g}^{\mathbbm{C}}$-module. By considering its weight space decomposition
\begin{center}
$\displaystyle{V = \bigoplus_{\mu \in \Phi(V)}V_{\mu}},$ \ \ \ \ 
\end{center}
such that $V_{\mu} = \{v \in V \ | \ \varrho(h)v = \mu(h)v, \ \forall h \in \mathfrak{h}\} \neq \{0\}$, $\forall \mu \in \Phi(V) \subset \mathfrak{h}^{\ast}$, we have the following definition.

\begin{definition}
\label{hightweightdef}
A highest weight vector (of weight $\lambda$) in a $\mathfrak{g}^{\mathbbm{C}}$-module $V$ is a non-zero vector $v_{\lambda}^{+} \in V_{\lambda}$, such that 
\begin{center}
$\varrho(x)v_{\lambda}^{+} = 0$, \ \ \ \ \ ($\forall x \in \mathfrak{n}^{+}$).
\end{center}
A weight $\lambda \in \Phi(V)$ associated with a highest weight vector is called highest weight of $V$.
\end{definition}

From above, we consider the following Cartan's highest weight theory results (e.g. \cite{Humphreys}):
\begin{enumerate}

\item[(A)] Every finite-dimensional irreducible $\mathfrak{g}^{\mathbbm{C}}$-module $V$ admits a highest weight vector $v_{\lambda}^{+}$. Moreover, $v_{\lambda}^{+}$ is the unique highest weight vector of $V$, up to non-zero scalar multiples.

\item[(B)] Let $V$ and $W$ be finite-dimensional irreducible $\mathfrak{g}^{\mathbbm{C}}$-modules with highest weight $\lambda \in \mathfrak{h}^{\ast}$. Then, $V$ and $W$ are isomorphic. 

\item[(C)] In the above setting, the following hold:
 
\begin{itemize}
\item[(C1)] If $V$ is an irreducible finite-dimensional  $\mathfrak{g}^{\mathbbm{C}}$-module with highest weight $\lambda \in \mathfrak{h}^{\ast}$, then $\lambda \in \Lambda^{+}$.

\item[(C2)] If $\lambda \in \Lambda^{+}$, then there exists an irreducible finite-dimensional  $\mathfrak{g}^{\mathbbm{C}}$-module $V$, such that $V = V(\lambda)$. 
\end{itemize}

\end{enumerate}
From item (C), it follows that the map $\lambda \mapsto V(\lambda)$ induces a one-to-one correspondence between $\Lambda^{+}$ and the set of isomorphism classes of finite-dimensional irreducible $\mathfrak{g}^{\mathbbm{C}}$-modules.

\begin{remark}
Given $\varpi_{\alpha} \in \Lambda^{+}$, we denote by $\Pi_{\varpi_{\alpha}} \colon G^{\mathbbm{C}} \to {\rm{GL}}(V(\varpi_{\alpha}))$ the associated irreducible representation.
\end{remark}

\subsection*{Structure of Parabolic Lie Groups}

Consider now the following definition.
\begin{definition}
A Lie subalgebra $\mathfrak{p} \subset \mathfrak{g}^{\mathbbm{C}}$ is called parabolic if $\mathfrak{p}$ contains some Borel subalgebra. A Lie subgroup $P \subset G^{\mathbbm{C}}$ is called parabolic if $P$ contains some Borel subgroup.
\end{definition}
Given $I \subset \Delta$, we can construct a parabolic Lie subalgebra by setting 
\begin{equation}
\mathfrak{p}_{I} := \mathfrak{b}(\Delta)\oplus \Big ( \sum_{\alpha \in \langle I \rangle^{-}} \mathfrak{g}_{\alpha}\Big ),
\end{equation}
here we denote
\begin{equation}
\langle I \rangle^{\pm} = \mathbbm{Z}I \cap \Phi^{\pm}, \ \ \ \ \langle I \rangle = \langle I \rangle^{+} \cup \langle I \rangle^{-}, \ \ \ \ \text{and} \ \ \ \ \Phi_{I}^{\pm} = \Phi^{\pm} \backslash \langle I \rangle^{\pm}.
\end{equation}
A Lie subalgebra as above is called standard parabolic Lie subalgebra relative to $\Delta$. Let $P_{I} \subset G^{\mathbbm{C}}$ be the Lie subgroup, such that ${\rm{Lie}}(P_{I}) = \mathfrak{p}_{I}$. By construction, $P_{I}$ is a parabolic Lie subgroup. In this last setting, we can show that $P_{I} = N_{G^{\mathbbm{C}}}(\mathfrak{p}_{I})$, where $N_{G^{\mathbbm{C}}}(\mathfrak{p}_{I})$ is the normalizer in  $G^{\mathbbm{C}}$ of $\mathfrak{p}_{I} \subset \mathfrak{g}^{\mathbbm{C}}$, see for instance \cite[\S 3.1]{Akhiezer}. From Theorem \ref{Borelconjugate}, we have the following result.
\begin{theorem}
Every parabolic Lie subalgebra is conjugate to one and only one standard parabolic Lie subalgebra $\mathfrak{p}_{I}$ relative to $\Delta$, for some $I \subset \Delta$.
\end{theorem}
From the result above, given a parabolic subgroup $P \subset G^{\mathbbm{C}}$, up to conjugation, we can always suppose that $P = P_{I}$, for some $I \subset \Delta$. From this, we observe that 
\begin{equation}
\label{Leviparabolic}
\mathfrak{p}_{I} = \underbrace{\Big ( \mathfrak{h} \oplus \sum_{\alpha \in \langle I \rangle }\mathfrak{g}_{\alpha}\Big)}_{\mathfrak{l}_{P}} \oplus \underbrace{\Big ( \sum_{\alpha \in \Phi_{I}^{+} }\mathfrak{g}_{\alpha}\Big )}_{\mathfrak{u}_{P}^{+}}.
\end{equation}
In the above decomposition (direct sum of vector spaces), $\mathfrak{l}_{P}$ is called the Levi factor and $\mathfrak{u}_{P}^{+}$ is called the nilpotent radical (maximal nilpotent ideal) of $\mathfrak{p}_{I}$. Notice that $\mathfrak{l}_{P}$ is a reductive Lie subalgebra, thus 
\begin{equation}
\mathfrak{l}_{P} =  \mathfrak{s}_{P}\oplus \mathfrak{z}(\mathfrak{l}_{P}), 
\end{equation}
where $\mathfrak{s}_{P}:= [\mathfrak{l}_{P},\mathfrak{l}_{P}]$ is semisimple and $\mathfrak{z}(\mathfrak{l}_{P})$ is the center of $\mathfrak{l}_{P}$, such that
\begin{equation}
\mathfrak{z}(\mathfrak{l}_{P}) = \bigcap_{\alpha\in I}\ker(\alpha),
\end{equation}
see for instance \cite[p. 68]{Akhiezer}. At the level of Lie groups, we have the Levi decomposition
\begin{equation}
\label{Levidecparabolic}
P_{I} = L_{P}R_{u}(P_{I})^{+} = L_{P} \ltimes R_{u}(P_{I})^{+},
\end{equation}
such that $L_{P} = S_{P}Z(L_{P})^{0}$, where $Z(L_{P})^{0}:= \exp(\mathfrak{z}(\mathfrak{l}_{P}))$ and $S_{P}= [L_{P},L_{P}]$. Further, since $ R_{u}(P_{I})^{+} \subset [P_{I},P_{I}]$, we obtain
\begin{equation}
\label{commutatordec}
P_{I} = [P_{I},P_{I}]Z(L_{P})^{0}.
\end{equation}
Further, denoting
\begin{equation}
T^{\mathbbm{C}}_{I} := \exp \bigg \{ \sum_{\alpha \in I}a_{\alpha}h_{\alpha} \ \bigg | \ a_{\alpha} \in \mathbbm{C}, \ \forall \alpha \in I\bigg\},
\end{equation}
it follows that the maximal torus $T^{\mathbbm{C}} := \exp(\mathfrak{h})$ of $G^{\mathbbm{C}}$ decomposes as follows
\begin{equation}
T^{\mathbbm{C}} = T^{\mathbbm{C}}_{I} Z(L_{P})^{0}.
\end{equation}
For more details on parabolic Lie groups, see for instance \cite[Part II]{jantzen2003representations}, \cite{HumphreysLAG}.\\

Given a parabolic Lie subgroup $P \subset G^{\mathbbm{C}}$, let $(L_{P},T^{\mathbbm{C}})$ be the split reductive group defined by its Levi component $L_{P} \subset P$. Considering the root datum $(\mathbbm{X}(T^{\mathbbm{C}}),\langle I \rangle, \alpha \mapsto \alpha^{\vee})$ of $(L_{P},T^{\mathbbm{C}})$, such that 
\begin{equation}
\label{charactertrosuabelian}
\mathbbm{X}(T^{\mathbbm{C}}) = {\rm{Hom}}(T^{\mathbbm{C}},\mathbbm{C}^{\times}) \cong \Lambda = \bigoplus_{\alpha \in \Delta} \mathbbm{Z}\varpi_{\alpha},
\end{equation}
is the character group of $T^{\mathbbm{C}}$, we can show that
\begin{equation}
\label{characterparabolic}
{\rm{Hom}}(P_{I},\mathbbm{C}^{\times})  \cong {\rm{Hom}}(L_{P},\mathbbm{C}^{\times})  \cong \Big \{ \vartheta \in \mathbbm{X}(T^{\mathbbm{C}}) \ \Big | \ \langle \vartheta, \alpha^{\vee}\rangle = 0, \ \ \forall \alpha \in I\Big \}.
\end{equation}
The isomorphism above is obtained from the restriction homomorphism
\begin{equation}
{\rm{res}} \colon {\rm{Hom}}(P_{I},\mathbbm{C}^{\times}) \to \mathbbm{X}(T^{\mathbbm{C}}), \ \ \chi \mapsto \chi|_{T^{\mathbbm{C}}}, 
\end{equation}
that is, ${\rm{Hom}}(P_{I},\mathbbm{C}^{\times}) \cong {\rm{Im}}({\rm{res}})$. For more details, see for instance \cite[Part II, p. 169]{jantzen2003representations}. 

\begin{remark}
From now on, we will not distinguish the isomorphic abelian groups described in Eq. (\ref{charactertrosuabelian}) and Eq. (\ref{characterparabolic}).
\end{remark}

\section{Rational Homogeneous Varieties}
\label{generalities}

In this section, we present some basic generalities on flag varieties. For more details on this subject, we suggest \cite{Akhiezer}, \cite{Flagvarieties}, \cite{HumphreysLAG}, \cite{BorelRemmert}, \cite{correa2026deformed}.

\subsection*{Generalities on Flag Varieties}

\begin{definition}
A complex flag variety $X$ is a compact simply connected homogeneous complex manifold defined by
\begin{equation}
X_{P} := G^{\mathbbm{C}}/P = G/G \cap P,
\end{equation}
where $G^{\mathbbm{C}}$ is a complex simple Lie group with compact real form given by $G$, and $P \subset G^{\mathbbm{C}}$ is a parabolic Lie subgroup. 
\end{definition}

Let us describe the Picard group of a flag variety $X_{P}$ defined by some parabolic Lie subgroup $P \subset G^{\mathbbm{C}}$. By choosing a trivializing open covering $X_{P} = \bigcup_{i \in J}U_{i}$, in terms of $\check{C}$ech cocycles we can write 
\begin{center}
$G^{\mathbbm{C}} = \Big \{(U_{i})_{i \in J}, \psi_{ij} \colon U_{i} \cap U_{j} \to P \Big \}$.
\end{center}
Since $P = P_{I}$, considering the characterization
\begin{equation}
{\rm{Hom}}(P_{I},\mathbbm{C}^{\times}) = \Big \{ \vartheta \in \mathbbm{X}(T^{\mathbbm{C}}) \ \Big | \ \langle \vartheta ,\alpha^{\vee} \rangle = 0, \ \ \forall \alpha \in I\Big \},
\end{equation}
for every $\alpha \in \Delta \backslash I$, from the homomorphism $\vartheta_{\varpi_{\alpha}} \colon P \to \mathbbm{C}^{\times}$, satisfying $(d\vartheta_{\varpi_{\alpha}})_{e} = \varpi_{\alpha}$, one can equip $\mathbbm{C}$ with the structure of a $P$-space, such that $pz = \vartheta_{\varpi_{\alpha}}(p)^{-1}z$, $\forall p \in P$, and $\forall z \in \mathbbm{C}$. Denoting by $\mathbbm{C}_{-\varpi_{\alpha}}$ this $P$-space, we can form an associated holomorphic line bundle $\mathscr{O}_{\alpha}(1) = G^{\mathbbm{C}} \times_{P}\mathbbm{C}_{-\varpi_{\alpha}}$, which can be described in terms of $\check{C}$ech cocycles by
\begin{equation}
\label{linecocycle}
\mathscr{O}_{\alpha}(1) = \Big \{(U_{i})_{i \in J},\vartheta_{\varpi_{\alpha}}^{-1} \circ \psi_{i j} \colon U_{i} \cap U_{j} \to \mathbbm{C}^{\times} \Big \},
\end{equation}
that is, $\mathscr{O}_{\alpha}(1) = \{g_{ij}\} \in \check{H}^{1}(X_{P},\mathcal{O}_{X_{P}}^{\ast})$, such that $g_{ij} = \vartheta_{\varpi_{\alpha}}^{-1} \circ \psi_{i j}$, $\forall i,j \in J$. 

Given $\mathscr{O}_{\alpha}(1) \in {\text{Pic}}(X_{P})$, such that $\alpha \in \Delta \backslash I$, as described above, if we consider an open covering $X_{P} = \bigcup_{i \in J} U_{i}$ which trivializes both $P \hookrightarrow G^{\mathbbm{C}} \to X_{P}$ and $ \mathscr{O}_{\alpha}(1) \to X_{P}$, by taking a collection of local sections $(s_{i})_{i \in J}$, such that $s_{i} \colon U_{i} \to G^{\mathbbm{C}}$, we can define $q_{i} \colon U_{i} \to \mathbbm{R}^{+}$, such that 
\begin{equation}
\label{functionshermitian}
q_{i}(x) := \frac{1}{||\Pi_{\varpi_{\alpha}}(s_{i}(x))v_{\varpi_{\alpha}}^{+}||^{2}},
\end{equation}
for every $i \in J$. Since $s_{j} = s_{i}\psi_{ij}$ on $U_{i} \cap U_{j} \neq \emptyset$, and $pv_{\varpi_{\alpha}}^{+} = \vartheta_{\varpi_{\alpha}}(p)v_{\varpi_{\alpha}}^{+}$, for every $p \in P$, and every $\alpha \in \Delta \backslash I$, the collection of functions $(q_{i})_{i \in J}$ satisfies $q_{j} = |\vartheta_{\varpi_{\alpha}}^{-1} \circ \psi_{ij}|^{2}q_{i}$ on $U_{i} \cap U_{j} \neq \emptyset$. Hence, we obtain a collection of functions $(q_{i})_{i \in J}$ which satisfies on the overlaps $U_{i} \cap U_{j} \neq \emptyset$ the following relation
\begin{equation}
\label{collectionofequ}
q_{j} = |g_{ij}|^{2}q_{i},
\end{equation}
such that $g_{ij} = \vartheta_{\varpi_{\alpha}}^{-1} \circ \psi_{i j}$, $\forall i,j \in J$. From this, we can define a Hermitian structure ${\bf{h}}$ on $\mathscr{O}_{\alpha}(1)$ by taking on each trivialization $f_{i} \colon \mathscr{O}_{\alpha}(1)|_{U_{i}} \to U_{i} \times \mathbbm{C}$ the metric defined by
\begin{equation}
\label{hermitian}
{\bf{h}}(f_{i}^{-1}(x,v),f_{i}^{-1}(x,w)) = q_{i}(x) v\overline{w},
\end{equation}
for every $(x,v),(x,w) \in U_{i} \times \mathbbm{C}$. The Hermitian metric above induces a Chern connection $\nabla \myeq {\rm{d}} + \partial \log {\bf{h}}$ with curvature $F_{\nabla}$ satisfying (locally)
\begin{equation}
\displaystyle \frac{\sqrt{-1}}{2\pi}F_{\nabla} \myeq \frac{\sqrt{-1}}{2\pi} \partial \overline{\partial}\log \Big ( \big | \big | (\Pi_{\varpi_{\alpha}} \circ s_{i})v_{\varpi_{\alpha}}^{+}\big | \big |^{2} \Big).
\end{equation}
Therefore, by considering the closed $G$-invariant $(1,1)$-form ${\bf{\Omega}}_{\alpha} \in \Omega^{1,1}(X_{P})^{G}$, which satisfies $\pi^{\ast}{\bf{\Omega}}_{\alpha} = \sqrt{-1}\partial \overline{\partial} \varphi_{\varpi_{\alpha}}$, where $\pi \colon G^{\mathbbm{C}} \to G^{\mathbbm{C}} / P = X_{P}$, and $\varphi_{\varpi_{\alpha}}(g) := \frac{1}{2\pi}\log \big (||\Pi_{\varpi_{\alpha}}(g)v_{\varpi_{\alpha}}^{+}||^{2} \big)$, for every $g \in G^{\mathbbm{C}}$, we have 
\begin{equation}
{\bf{\Omega}}_{\alpha} |_{U_{i}} = (\pi \circ s_{i})^{\ast}{\bf{\Omega}}_{\alpha} = \frac{\sqrt{-1}}{2\pi}F_{\nabla} \Big |_{U_{i}},
\end{equation}
i.e., $c_{1}(\mathscr{O}_{\alpha}(1)) = [ {\bf{\Omega}}_{\alpha}]$, $\forall \alpha \in \Delta \backslash I$. From the above construction, the following result summarizes the main properties to be considered in this work about invertible coherent sheaves and real $G$-invariant $(1,1)$-forms on flag varieties.
\begin{theorem}
\label{AZADBISWAS}
Given a flag variety $X_{P} = G^{\mathbbm{C}}/P$, such that $P = P_{I}$, for some $I \subset \Delta$, then the following hold:
\begin{enumerate}
\item[(1)] As an abelian group, the Picard group of $X_{P}$ is generated by $\mathscr{O}_{\alpha}(1), \alpha \in \Delta \backslash I$, i.e.,
\begin{equation}
{\rm{Pic}}(X_{P}) = \big \langle \mathscr{O}_{\alpha}(1) \ | \ \alpha \in \Delta \backslash I \big \rangle_{\mathbbm{Z}}.
\end{equation}
\item[(2)] $H^{2}(X_{P},\mathbbm{Z}) = \bigoplus_{\alpha \in \Delta \backslash I}\mathbbm{Z}[{\bf{\Omega}}_{\alpha}]$, such that $c_{1}(\mathscr{O}_{\alpha}(1)) = [{\bf{\Omega}}_{\alpha}], \forall \alpha \in \Delta \backslash I$.
\item[(3)] $\forall \alpha \in \Delta \backslash I$, we have $\pi^{\ast}{\bf{\Omega}}_{\alpha} = \sqrt{-1}\partial \overline{\partial} \varphi_{\varpi_{\alpha}}$, such that $\varphi_{\varpi_{\alpha}} \colon G^{\mathbbm{C}} \to \mathbbm{R}$ is given by
\begin{equation}
\varphi_{\varpi_{\alpha}}(g)  = \frac{1}{\pi}\log \big (||\Pi_{\varpi_{\alpha}}(g)v_{\varpi_{\alpha}}^{+}|| \big ), \ \ \forall g \in G^{\mathbbm{C}},
\end{equation}
where $\pi \colon G^{\mathbbm{C}} \to G^{\mathbbm{C}} / P = X_{P}$ is the natural projection.
\item[(4)] The K\"{a}hler cone $\mathcal{K}(X_{P})$ of $X_{P}$ is given explicitly by $\mathcal{K}(X_{P}) = \displaystyle \bigoplus_{\alpha \in \Delta \backslash I} \mathbbm{R}_{+}[ {\bf{\Omega}}_{\alpha}]$.
\item[(5)] Consider $\mathbbm{P}_{\beta}^{1} = \overline{\exp(\mathfrak{g}_{-\beta}){o}} \subset X_{P}$, such that $o = eP$ and $\beta \in \Phi_{I}^{+}$. Then, 
\begin{equation}
\int_{\mathbbm{P}_{\beta}^{1}} {\bf{\Omega}}_{\alpha} = \langle \varpi_{\alpha}, \beta^{\vee}  \rangle, \ \forall \alpha \in \Delta \backslash I.
\end{equation}
In particular, the cone of curves ${\rm{NE}}(X_{P})$ is given by ${\rm{NE}}(X_{P}) = \sum_{\alpha \in \Delta \backslash I} \mathbbm{R}_{\geq 0}[\mathbbm{P}_{\alpha}^{1}]$.
\end{enumerate}
\end{theorem}
The proof of the above theorem follows from \cite{AZAD}, see also \cite{correa2026deformed}, \cite{FultonWoodward}. 
\begin{remark}
\label{powergenerators}
From the previous theorem, for every $\alpha \in \Delta \backslash I$, we denote 
\begin{equation}
\mathscr{O}_{\alpha}(\ell) := \mathscr{O}_{\alpha}(1)^{\otimes \ell}, \ \ \forall \ell \in \mathbbm{Z}.
\end{equation}
Notice that $\mathscr{O}_{\alpha}(0) := \mathcal{O}_{X_{P}}$, for every $\alpha \in \Delta \backslash I$.
\end{remark}
\begin{remark}
\label{harmonic2forms}
Given any $G$-invariant Riemannian metric $g$ on $X_{P}$, denoting by $\mathscr{H}^{2}(X_{P},g)$ the space of real harmonic 2-forms on $X_{P}$ with respect to $g$, then
\begin{equation}
\mathscr{H}^{2}(X_{P},g) = \mathscr{I}_{G}^{1,1}(X_{P}),
\end{equation}
where $\mathscr{I}_{G}^{1,1}(X_{P})$ is the space of closed $G$-invariant real $(1,1)$-forms on $X_{P}$, for more details, see for instance \cite[Lemma 3.1]{MR528871}.
\end{remark}

In the above setting, we denote the weights of $P = P_{I}$ by  
\begin{center}
$\displaystyle \Lambda_{P} := \bigoplus_{\alpha \in \Delta \backslash I}\mathbbm{Z}\varpi_{\alpha}$. 
\end{center}
From this, the previous theorem provides $\Lambda_{P} \cong {\rm{Hom}}(P,\mathbbm{C}^{\times}) \cong {\rm{Pic}}(X_{P})$, such that
\begin{enumerate}
\item$ \displaystyle \lambda = \sum_{\alpha \in \Delta \backslash I}k_{\alpha}\varpi_{\alpha} \mapsto \prod_{\alpha \in \Delta \backslash I} \vartheta_{\varpi_{\alpha}}^{k_{\alpha}} \mapsto \bigotimes_{\alpha \in \Delta \backslash I} \mathscr{O}_{\alpha}(k_{\alpha})$.
\item $ \displaystyle {\bf{E}} \mapsto \vartheta_{{\bf{E}}}: = \prod_{\alpha \in \Delta \backslash I} \vartheta_{\varpi_{\alpha}}^{\langle c_{1}({\bf{E}}),[\mathbbm{P}^{1}_{\alpha}] \rangle} \mapsto \lambda({\bf{E}}) := \sum_{\alpha \in \Delta \backslash I}\langle c_{1}({\bf{E}}),[\mathbbm{P}^{1}_{\alpha}] \rangle\varpi_{\alpha}$.
\end{enumerate}
Thus, $\forall {\bf{E}} \in {\rm{Pic}}(X_{P})$, we have $\lambda({\bf{E}}) \in \Lambda_{P}$. More generally, $\forall \xi \in H^{1,1}(X_{P},\mathbbm{R})$, we can associate $\lambda(\xi) \in \Lambda_{P}\otimes \mathbbm{R}$, such that
\begin{equation}
\label{weightcohomology}
\lambda(\xi) := \sum_{\alpha \in \Delta \backslash I}\langle \xi,[\mathbbm{P}^{1}_{\alpha}] \rangle\varpi_{\alpha}.
\end{equation}

\begin{remark}
From now on, given $I \subset \Delta$, we denote $\Phi_{I}^{\pm}:= \Phi^{\pm} \backslash \langle I \rangle^{\pm}$.
\end{remark}

\begin{remark}
\label{bigcellcosntruction}
From above, if ${\bf{L}} \in {\rm{Pic}}(X_{P})$ is a very ample line bundle, from the Borel-Weil theorem  we have a projective embedding 
\begin{equation}
X_{P} \hookrightarrow \mathbbm{P}(H^{0}(X_{P},{\bf{L}})^{\ast}) \cong  \mathbbm{P}(V(\lambda({\bf{E}})))) 
\end{equation}
It is worth mentioning that on a flag variety $X_{P}$ we have a distinguished neighborhood 
\begin{center}
$o = eP  \in U^{-}(P) \subset X_{P}$
\end{center}
defined as follows
\begin{equation}
\label{bigcell}
 U^{-}(P) =  B^{-}o = R_{u}(P_{I})^{-}o\subset X_{P},  
\end{equation}
 where $B^{-} = \exp(\mathfrak{h} \oplus \mathfrak{n}^{-})$, and
 
 \begin{center}
 
 $R_{u}(P_{I})^{-} = \displaystyle \prod_{\alpha \in \Phi_{I}^{+}}N_{\alpha}^{-}$, \ \ (opposite unipotent radical)
 
 \end{center}
with $N_{\alpha}^{-} = \exp(\mathfrak{g}_{-\alpha})$, $\forall \alpha \in \Phi_{I}^{+}$, e.g. \cite[\S 3]{Lakshmibai2}, \cite[\S 3.1]{Akhiezer}. As we see, the opposite big cell defines a contractible open dense subset in $X_{P}$, thus the restriction of any vector bundle (principal bundle) over this open set is trivial. Furthermore, since $R_{u}(P_{I})^{-}$ is an algebraic group isomorphic to the affine space, we conclude that $X_{P}$ is a rational homogeneous variety.
\end{remark}

Now we consider the following result, see for instance \cite{correa2026deformed}.

\begin{proposition}
\label{eigenvalueatorigin}
Let $X_{P}$ be a flag variety and let $\omega$ be a $G$-invariant K\"{a}hler metric on $X_{P}$. Then, for every closed $G$-invariant real $(1,1)$-form $\theta$, the eigenvalues of the endomorphism $\omega^{-1} \circ \theta$ are given by 
\begin{equation}
\label{eigenvalues}
{\bf{q}}_{\beta}(\omega^{-1} \circ \theta) = \frac{ \langle \lambda([\theta]), \beta^{\vee} \rangle}{\langle \lambda([\omega]), \beta^{\vee} \rangle}, \ \ \beta \in \Phi_{I}^{+},
\end{equation}
such that $\lambda([\theta]), \lambda([\omega]) \in \Lambda_{P} \otimes \mathbbm{R}$.
\end{proposition}

\begin{remark}
\label{normalizationform}
In the setting of the proof of Proposition \ref{eigenvalueatorigin}, we obtain from the previous result the following pointwise description  
\begin{equation}
 \omega = \sum_{\beta \in \Phi_{I}^{+}} \frac{\sqrt{-1}}{2} dz_{\beta} \wedge d\overline{z_{\beta}} \ \ \ \ {\text{and}} \ \ \ \ \displaystyle \theta = \sum_{\beta \in \Phi_{I}^{+}} \frac{\sqrt{-1}}{2} {\bf{q}}_{\beta}(\omega^{-1} \circ \theta)dz_{\beta} \wedge d\overline{z_{\beta}},
\end{equation}
for every closed $G$-invariant real $(1,1)$-form $\theta \in \Omega^{1,1}(X_{P})$.
\end{remark}

\begin{remark}
\label{primitivecalc}
Given ${\bf{\Omega}}_{\alpha} \in c_{1}(\mathscr{O}_{\alpha}(1))$, $\alpha \in \Delta \backslash I$, and fixed some $G$-invariant K\"{a}hler metric $\omega$ on $X_{P}$, since ${\bf{\Omega}}_{\alpha}$ is harmonic with respect to $\omega$ (see Remark \ref{harmonic2forms}), it follows that 
\begin{equation}
-{\rm{d}}^{c}\Lambda_{\omega}({\bf{\Omega}}_{\alpha}) = \delta_{\omega}{\bf{\Omega}}_{\alpha} = 0,
\end{equation}
i.e., $\Lambda_{\omega}({\bf{\Omega}}_{\alpha})$ is constant. Thus, we obtain
\begin{equation}
\label{contraction}
\Lambda_{\omega}({\bf{\Omega}}_{\alpha})= {\rm{tr}}(\omega^{-1} \circ {\bf{\Omega}}_{\alpha}) = \sum_{\beta \in \Phi_{I}^{+}} \frac{\langle \varpi_{\alpha}, \beta^{\vee} \rangle}{\langle \lambda([\omega]), \beta^{\vee}\rangle},
\end{equation}
for every $\alpha \in \Delta \backslash I$. In particular, for every holomorphic line bundle ${\bf{E}} \in {\rm{Pic}}(X_{P})$, we have a Hermitian structure ${\bf{h}}$ on ${\bf{E}}$, such that the curvature $F_{\nabla}$ of the associated Chern connection $\nabla \myeq {\rm{d}} + \partial \log ({\bf{h}})$, satisfies 
\begin{equation}
\label{tracecurvature}
\frac{\sqrt{-1}}{2\pi} \Lambda_{\omega}(F_{\nabla}) = \sum_{\beta \in \Phi_{I}^{+} } \frac{\langle \lambda({\bf{E}}), \beta^{\vee} \rangle}{\langle \lambda([\omega]), \beta^{\vee}\rangle}.
\end{equation}
From this, we have that $\nabla$ is a Hermitian-Yang-Mills (HYM) connection (e.g. \cite{Kobayashi+1987}). Notice that 
\begin{equation}
c_{1}({\bf{E}}) = \sum_{\alpha \in \Delta \backslash I}\langle \lambda({\bf{E}}), \alpha^{\vee} \rangle [{\bf{\Omega}}_{\alpha}],
\end{equation}
for every ${\bf{E}} \in {\rm{Pic}}(X_{P})$, i.e., the curvature of the HYM connection $\nabla$ on ${\bf{E}}$ coincides with the $G$-invariant representative of $c_{1}({\bf{E}})$. 
\end{remark}

\begin{remark}
\label{proj}
It is worthwhile to observe that, in the setting above, if ${\bf{L}} \in {\text{Pic}}(X_{P})$ is a negative line bundle, then ${\bf{L}}^{-1}$ is in fact very ample, i.e., we have a projective embedding 
\begin{center}
$\iota \colon X_{P} \hookrightarrow \mathbbm{P}(V(\lambda({\bf{L^{-1}}}))) = {\text{Proj}}\big (H^{0}(X_{P},{\bf{L}}^{-1})^{\ast} \big).$
\end{center}
Thus, we have the identification
\begin{equation}
\label{tautological}
{\bf{L}} \cong  \Big \{ \big ([x],v \big ) \in X_{P} \times V(\lambda({\bf{L}}^{-1})) \ \ \Big | \ \ v \in \langle x \rangle_{\mathbb{C}}\Big \},
\end{equation}
Here we consider the geometric realization $V(\lambda({\bf{L}}^{-1})) \cong  H^{0}(X_{P},{\bf{L}}^{-1})^{\ast}$ via Borel-Weil theorem (e.g. \cite{serre1954representations}, \cite{sale2002several}, \cite{Akhiezer}).
\end{remark}

\subsection*{Canonical Bundle of Flag Varieties} By keeping the previous notation, let $X_{P} = G^{\mathbbm{C}}/P$ be a complex flag variety defined by some parabolic Lie subgroup $P = P_{I} \subset G^{\mathbbm{C}}$. Considering the identification $T_{{\rm{o}}}^{1,0}X_{P} \cong \mathfrak{m} \subset \mathfrak{g}^{\mathbbm{C}}$, such that ${\rm{o}} = eP \in X$ and 
\begin{center}
$\mathfrak{m} = \displaystyle \sum_{\alpha \in \Phi_{I}^{-}} \mathfrak{g}_{\alpha}$,
\end{center}
 we can realize $T^{1,0}X$ as being a holomoprphic vector bundle, associated with the holomorphic principal $P$-bundle $P \hookrightarrow G^{\mathbbm{C}} \to X_{P}$, such that 

\begin{center}

$T^{1,0}X_{P} = \Big \{(U_{i})_{i \in J}, \underline{{\rm{Ad}}}\circ \psi_{i j} \colon U_{i} \cap U_{j} \to {\rm{GL}}(\mathfrak{m}) \Big \}$,

\end{center}
where $\underline{{\rm{Ad}}} \colon P \to {\rm{GL}}(\mathfrak{m})$ is the isotropy representation. From this, we obtain 
\begin{equation}
\label{canonicalbundleflag}
{\bf{K}}_{X_{P}}^{-1} = \det \big(T^{1,0}X_{P} \big) = \Big \{(U_{i})_{i \in J}, \det (\underline{{\rm{Ad}}}\circ \psi_{i j}) \colon U_{i} \cap U_{j} \to \mathbbm{C}^{\times} \Big \}.
\end{equation}
Since the character $\det \circ \underline{{\rm{Ad}}} \in {\text{Hom}}(P,\mathbbm{C}^{\times})$ is completely determined by its restriction to the torus ${\bf{Z}}(L_{P})^{0} = \exp(\mathfrak{z}(\mathfrak{l}_{P})) \subset P$, see Eq. (\ref{commutatordec}), it follows that 
\begin{equation}
\det \underline{{\rm{Ad}}}(\exp({\bf{s}})) = {\rm{e}}^{{\rm{tr}}({\rm{ad}}({\bf{s}})|_{\mathfrak{m}})} = {\rm{e}}^{- \langle \delta_{P},{\bf{s}}\rangle },
\end{equation}
$\forall {\bf{s}} \in \mathfrak{z}(\mathfrak{l}_{P})$, such that 
\begin{equation}
\delta_{P} = \sum_{\alpha \in \Phi_{I}^{+} } \alpha =  \sum_{\alpha \in \Delta \backslash I} \langle \delta_{P}, \alpha^{\vee} \rangle \varpi_{\alpha}. 
\end{equation}
In fact, considering the representation 
\begin{center}
${\rm{ad}}(\cdot)|_{\mathfrak{u}_{P}^{+}} \colon \mathfrak{l}_{P} \to \mathfrak{gl}(\mathfrak{u}_{P}^{+})$, 
\end{center}
for every $\alpha \in I$, we have
\begin{equation}
\langle \delta_{P},\alpha^{\vee} \rangle = \delta_{P}(h_{\alpha}) = {\rm{Tr}}({\rm{ad}}(h_{\alpha})|_{\mathfrak{u}^{+}_{P}}) = 0,
\end{equation}
since $h_{\alpha}  \in \mathfrak{s}_{P} = [\mathfrak{l}_{P},\mathfrak{l}_{P}]$. From this, we conclude that 
\begin{equation}
{\bf{K}}_{X_{P}} = \bigotimes_{\alpha \in \Delta \backslash I}\mathscr{O}_{\alpha}(-1)^{\langle \delta_{P},\alpha^{\vee} \rangle},
\end{equation}
see Remark \ref{powergenerators}. From above, we see that $X_{P}$ is a Fano variety.

\begin{remark}
\label{canonicalflag}
If we consider the invariant K\"{a}hler metric $\omega_{0} \in \Omega^{1,1}(X_{P})^{G}$ defined by
\begin{equation}
\label{riccinorm}
\omega_{0} := \sum_{\alpha \in \Delta \backslash I}2 \pi \langle \delta_{P}, \alpha^{\vee} \rangle {\bf{\Omega}}_{\alpha},
\end{equation}
it follows that
\begin{equation}
\label{ChernFlag}
c_{1}(X_{P}) = - c_{1}({\bf{K}}_{X_{P}}) = \Big [ \frac{\omega_{0}}{2\pi}\Big].
\end{equation}
By the uniqueness of $G$-invariant representative of $c_{1}(X_{P})$, we conclude that  
\begin{center}
\label{Ricciinvariant}
${\rm{Ric}}(\omega) = \omega_{0}$, 
\end{center}
for every $G$-invariant Kähler metric $\omega$ on $X_{P}$. In particular, $\omega_{0} \in \Omega^{1,1}(X_{P})^{G}$ defines a $G$-invariant K\"{a}hler-Einstein metric on $X_{P}$ (cf. \cite{MATSUSHIMA}).

In the above setting, since ${\bf{K}}_{X_{P}}^{-1}$ is a very ample line bundle, we have the projective embedding 
\begin{equation}
X_{P} \hookrightarrow \mathbbm{P}(H^{0}(X,{\bf{K}}_{X_{P}}^{-1})^{\ast}) = \mathbbm{P}(V(\delta_{P})),
\end{equation}
see Remark \ref{proj}. In particular, considering the associated affine cone 
\begin{equation}
{\rm{Aff}}(X_{P}) := {\rm{Spec}} \Bigg (\bigoplus_{n\geq 0}H^{0}\big (X,({\bf{K}}_{X_{P}}^{-1})^{\otimes n} \big ) \Bigg)
\end{equation}
it follows that the Cartan-Remmert reduction (e.g. \cite{grauert1962modifikationen})
\begin{equation}
\mathscr{R} \colon Z = {\rm{Tot}}({\bf{K}}_{X_{P}}) \to {\rm{Aff}}(X_{P}),
\end{equation}
provides a resolution of the isolated singularity at the vertex of the affine cone ${\rm{Aff}}(X_{P})$. This resolution is a crepant resolution of singularity, i.e., ${\bf{K}}_{Z} = \mathscr{R}^{\ast} {\bf{K}}_{{\rm{Aff}}(X_{P})}$, see for instance \cite{van2010ricci} and \cite{Correa}. In particular, since 
\begin{equation}
{\bf{K}}_{Z} \cong p^{\ast}({\bf{K}}_{X_{P}} \otimes {\bf{K}}_{X_{P}}^{-1}) \cong \mathcal{O}_{Z},
\end{equation}
it follows that ${\rm{Aff}}(X_{P})$ is a singular Calabi-Yau variety.

For the sake of completeness, let us describe the holomorphic volume form on ${\bf{K}}_{X_{P}}$.  Let $(z_{1}, \dots, z_{n})$ be local coordinates on a chart $U_{\alpha}$ of the base variety $X_{P}$. A local holomorphic top-form on $X_{P}$ is given by
\begin{equation}
\tau_{\alpha} = dz_{1} \wedge \dots \wedge dz_{n}
\end{equation}
The form $\tau_{\alpha}$ serves as a local section of the canonical bundle ${\bf{K}}_{X_{P}}$. Hence, we can use $\tau_{\alpha}$ as a local frame for the fibers. Let $w_{\alpha}$ be the fiber coordinate relative to this choice of frame, meaning a point in the total space of ${\bf{K}}_{X_{P}}$ is locally represented as $w_{\alpha} \cdot \tau_{\alpha}$.

Consider another chart $U_{\beta}$ with local coordinates $(z'_{1}, \dots, z'_{n})$ and local top-form 
\begin{center}
$\tau_\beta = dz'_{1} \wedge \dots \wedge dz'_{n}$. 
\end{center}
The base top-forms transform via the Jacobian determinant $J_{\alpha\beta}$, that is, $\tau_{\alpha} = J_{\alpha\beta} \tau_{\beta}$. Since the geometric point in the total space must be independent of the choice of frame ($w_{\alpha} \tau_{\alpha} = w_{\beta} \tau_{\beta}$), the fiber coordinates must transform inversely as $w_{\alpha} = J_{\alpha\beta}^{-1} w_{\beta}$. To construct the global top-degree $(n+1)$-form $\Omega$ on the total space ${\bf{K}}_{X_{P}}$, we take the exterior differential of the coordinates. In the $\alpha$-chart, the natural candidate is $\Omega_{\alpha} = \tau_{\alpha} \wedge dw_{\alpha}$. Let us check how this transforms into the $\beta$-chart:
\begin{equation}
dw_{\alpha} = d(J_{\alpha\beta}^{-1} w_{\beta}) = J_{\alpha\beta}^{-1} dw_{\beta} + w_{\beta}d(J_{\alpha\beta}^{-1})
\end{equation}
Substituting this and the transformation of $\tau_\alpha$ yields
\begin{equation}
\tau_\alpha = (J_{\alpha\beta} \tau_\beta) \wedge ( J_{\alpha\beta}^{-1} dw_{\beta} + w_{\beta}d(J_{\alpha\beta}^{-1}) )
\end{equation}
Expanding the wedge product gives:
\begin{equation}
\Omega_{\alpha} = \tau_{\beta} \wedge dw_{\beta} + w_{\beta} J_{\alpha\beta}\tau_{\beta} \wedge d(J_{\alpha\beta}^{-1})
\end{equation}
The second term vanishes identically because $d(J_{\alpha\beta}^{-1}) = \sum \partial_{z'_{i}} (J_{\alpha\beta}^{-1}) dz'_{i}$, and wedging any $dz'_{i}$ with the already complete base top-form $\tau_{\beta} = dz'_{1} \wedge \dots \wedge dz'_{n}$ yields zero. Thus, we obtain $\Omega_{\alpha} = \tau_{\beta} \wedge dw_{\beta} = \Omega_{\beta}$ on the overlap $U_{\alpha} \cap U_{\beta}$. This proves that the $(n+1)$-form 
\begin{equation}
\Omega = dz_{1} \wedge \dots \wedge dz_{n} \wedge dw
\end{equation}
is globally well-defined, independent of the choice of local coordinates, and strictly non-vanishing everywhere on the total space of $Z = {\rm{Tot}}({\bf{K}}_{X_{P}})$.

\end{remark}

\subsection*{dHYM Equation on Flag Varieties}

Let $(X, \omega)$ be a compact Kähler manifold of complex dimension $n$. The dHYM equation for a closed form $\chi \in [\psi] \in H^{1,1}(X, \mathbbm{R})$ seeks a solution to the equation
\begin{equation}
    {\rm{Im}}(\omega + \sqrt{-1}\chi)^{n} = \tan(\hat{\Theta}) {\rm{Re}}(\omega + \sqrt{-1}\chi)^{n},
\end{equation}
where $\hat{\Theta}$ is a topological phase constant. Equivalently, one seeks $\chi$ such that
\begin{equation}
\Theta_{\omega}(\chi) = \sum_{j=1}^{n} \arctan({\bf{q}}_{j}) = \hat{\Theta} \ ({\rm{mod}} \ 2\pi)
\end{equation}
is constant, where ${\bf{q}}_{j}$ are the eigenvalues of the endomorphism $\omega^{-1} \circ \chi$. In the above setting, we say that $\Theta_{\omega}(\chi)$ is the Lagrangian phase of $\chi$.

The existence problem for the dHYM equation can be solved unconditionally in the context of rational homogeneous varieties $X_{P} = G/P$, more precisely, we have the following theorem.

\begin{theorem}[\cite{correa2026deformed}]
\label{theoremDHYM}
Given a  K\"{a}hler class $[\omega] \in \mathcal{K}(X_{P})$, for every $[\eta] \in H^{1,1}(X_{P},\mathbbm{R})$ we have 
\begin{equation}
\label{phaseangleCartan}
\hat{\Theta}= {\textrm{Arg}} \int_{X_{P}}\frac{(\omega + \sqrt{-1}\eta)^{n}}{n!} = \sum_{\beta \in \Phi_{I}^{+}} \arctan \bigg( \frac{\langle \lambda([\eta]),\beta^{\vee} \rangle}{\langle \lambda([\omega]),\beta^{\vee} \rangle}\bigg) \ ({\textrm{mod}}\ 2 \pi),
\end{equation}
such that $\lambda([\eta]), \lambda([\omega]) \in \Lambda_{P} \otimes \mathbbm{R}$. In particular, fixed the unique $G$-invariant representative $\omega_{0} \in [\omega]$, there exists $f \in C^{\infty}(X_{P})$ such that $\chi := \eta + \sqrt{-1}\partial \overline \partial f$ satisfies the deformed Hermitian Yang-Mills equation
\begin{equation}
\label{DHYMeqTeo}
{\rm{Im}}\big ( \omega_{0} + \sqrt{-1}\chi \big)^{n} = \tan(\hat{\Theta}) {\rm{Re}}\big ( \omega_{0} + \sqrt{-1}\chi\big)^{n}.
\end{equation}
\end{theorem}

In the above theorem we have that $\chi$ is the unique $G$-invariant representative in the cohomology class $[\psi]$. Moreover, the Lagrangian phase of $\chi$ with respect to $\omega_{0}$ is given by
\begin{equation}
\Theta_{\omega_{0}}(\chi) = \sum_{\beta \in \Phi_{I}^{+}} \arctan \bigg( \frac{\langle \lambda([\chi]),\beta^{\vee} \rangle}{\langle \lambda([\omega_{0}]),\beta^{\vee} \rangle}\bigg).
\end{equation}
In particular, given ${\bf{L}} \in {\rm{Pic}}(X_{P})$, we have 
\begin{equation}
\label{lagrangianphaselinebundle}
\hat{\Theta}({\bf{L}}) = \displaystyle {\textrm{Arg}} \int_{X_{P}}\frac{(\omega_{0} + \sqrt{-1}\chi)^{n}}{n!} = \Theta_{\omega_{0}}(\chi)  \ ({\textrm{mod}}\ 2 \pi), 
\end{equation}
where $\chi \in  c_{1}({\bf{L}})$ is the unique $G$-invariant representative. Thus, we set $\hat{\Theta}({\bf{L}}):= \Theta_{\omega_{0}}(\chi)$.

\section{Calabi-Yau Metrics via Calabi-Ansatz}

In this section, we recall the construction of Calabi-Yau metrics on the canonical bundle of K\"{a}hler-Einstein Fano manifolds, for more details on this subject, see \cite{Calabi1979} and \cite{HwangSinger2002}.

Since our interest is in the particular class of such manifolds provided by rational homogeneous varieties, based on the previous results, we shall consider from now on that $X = G^{\bf{C}}/P$ is a flag variety defined by a parabolic subgroup $P = P_{I} \subset G^{\bf{C}}$. In this setting, we fix a $G$-invariant K\"{a}hler-Einstein metric $\omega_{0}$ on $X$, such that 
\begin{center}
${\rm{Ric}}(\omega_{0}) = \omega_{0}$.
\end{center}
Due to the triviality of the first Chern class of $Z = {\rm{Tot}}({\bf{K}}_{X})$, this space admits a complete Ricci-flat Kähler metric. The construction of this metric relies on the celebrated Calabi ansatz. In order to establish this result, we consider the following auxiliary definition and proposition.

\begin{definition}
A divergent path on a Riemannian manifold $(M,g)$ is a continuous curve $\gamma\colon  [0, 1] \to M$ such that, for any compact subset $K$ of $M$, there is a number $t_{0}(K)$ such that $\gamma(t)$ is contained in the complement $M \backslash K$ for all $t > t_{0}(K)$. In other words, a divergent path on $M$ is a ray that ultimately leaves every compact subset of $M$.
\end{definition}

\begin{proposition}
\label{completemanifold}
A Riemannian manifold $(M,g)$ is complete if and only if every divergent $C^{1}$-path $\gamma \colon [0, 1) \to M$ has infinite length.
\end{proposition}

\begin{proof}[{\bf{Proof}}]
$(\Rightarrow)$ If $M$ is complete and $\gamma \colon [0, 1) \to M$ is an arbitrary $C^1$-path of finite length, then $\gamma([0, 1))$ is bounded. Consequently, the closure of $\gamma([0, 1))$ is compact by the Hopf--Rinow theorem, and therefore $\gamma$ is not divergent. 

($\Leftarrow$) If $M$ is not complete, then we can find a geodesic $\gamma \colon [0, 1) \to M$ having $[0, 1)$ as its maximal domain of definition. The curve $\gamma$ is divergent since otherwise $\lim_{t \to 1^{-}} \gamma(t)$ would exist and $\gamma(t)$ could be extended beyond $t = 1$. Since $\gamma$ is a geodesic, its speed $\|\dot{\gamma}(t)\|$ is constant, $\forall t \in [0, 1)$, and therefore the length $l(\gamma) = \int_{0}^{1} ||\dot{\gamma}(t)|| dt$ of $\gamma$ is finite.
\end{proof}

\begin{theorem}[Calabi]
There exists a complete, Ricci-flat K\"{a}hler metric on the manifold $Z = {\rm{Tot}}({\bf{K}}_{X})$ of the form 
\begin{equation}
    \omega_{{\rm{CY}}} = p^{\ast}\omega_{0} + \sqrt{-1} \partial \bar{\partial} F(s),
\end{equation}
where $F$ is a smooth real-valued function depending strictly on the radial coordinate $s$.
\end{theorem}

\begin{proof}[{\bf{Proof}}]
Let $p: {\bf{K}}_{X} \to X$ be the natural projection. The K\"{a}hler-Einstein metric $\omega_{0}$ induces a Hermitian metric $h$ on the fibers of ${\bf{K}}_{X}$ whose curvature form satisfies 
\begin{equation}
\label{curveKE}
\sqrt{-1}F_{h} = \sqrt{-1}\bar{\partial} \partial\log h = - \omega_{0}. 
\end{equation}
Let $s = r^{2}$ denote the squared norm of a fiber element with respect to $h$. We carry out the geometric computations on the punctured bundle $Z^{\times} = {\rm{Tot}}({\bf{K}}_{X}^{\times})$, where $s > 0$. Using the curvature relation given in Eq. (\ref{curveKE}), a straightforward calculation in a local trivialization yields
\begin{equation}
    \sqrt{-1} \partial \bar{\partial} s = s  p^{\ast}\omega_{0} + \frac{\sqrt{-1} \partial s \wedge \bar{\partial} s}{s}.
\end{equation}
Applying the chain rule to $F(s)$, we obtain
\begin{equation}
    \sqrt{-1} \partial \bar{\partial} F(s) = F'(s) \sqrt{-1} \partial \bar{\partial} s + F''(s) \sqrt{-1} \partial s \wedge \bar{\partial} s.
\end{equation}
Substituting this result into the ansatz for $\omega_{{\rm{CY}}}$, we can group the horizontal and vertical components as follows
\begin{equation}
    \omega_{{\rm{CY}}} = (1 + sF'(s)) p^{\ast}\omega_{0} + (F'(s) + sF''(s)) \frac{\sqrt{-1} \partial s \wedge \bar{\partial} s}{s}.
\end{equation}
We define the geometric scale factors for the base and the fiber, respectively, as 
\begin{equation}
U(s) := 1 + sF'(s) \ \ \  \text{and} \ \ \ V(s) := F'(s) + sF''(s). 
\end{equation}
Note the fundamental identity $V(s) = U'(s)$. Thus, the volume form of $\omega_{{\rm{CY}}}$ is
\begin{equation}
    \omega_{{\rm{CY}}}^{n+1} = (n+1) U(s)^{n} U'(s) p^{\ast}\omega_{0}^{n} \wedge \frac{\sqrt{-1} \partial s \wedge \bar{\partial} s}{s}.
\end{equation}
To evaluate the Ricci curvature, we expand the wedge product in local trivializing coordinates $(x, w)$, where $x \in X$ and $w$ is the fiber coordinate. Since $s = h(x)|w|^{2}$, we have 
\begin{center}
$\displaystyle \partial \log s = \partial \log h + \frac{dw}{w}$. 
\end{center}
Because $p^{\ast}\omega_{0}^{n}$ is a top form on the base, the mixed terms vanish, yielding
\begin{equation}
    p^{\ast}\omega_{0}^{n} \wedge \frac{\sqrt{-1} \partial s \wedge \bar{\partial} s}{s} = s \cdot p^{\ast}\omega_{0}^{n} \wedge \big(\sqrt{-1} \partial \log s \wedge \bar{\partial} \log s\big) = h p^{\ast}\omega_{0}^{n} \wedge \sqrt{-1} dw \wedge d\bar{w}.
\end{equation}
By the K\"{a}hler-Einstein condition ${\rm{Ric}}(\omega_{0}) = \omega_{0}$, the volume form of the base satisfies 
\begin{center}
$\sqrt{-1}\partial\bar{\partial}\log(\omega_{0}^{n}) = -\omega_{0}$. 
\end{center}
Simultaneously, Eq. (\ref{curveKE}) implies $\sqrt{-1}\partial\bar{\partial}\log h = \omega_{0}$. As the determinant of the base $\omega_{0}^{n}$ and the metric $h$ transform reciprocally, the product $h  \omega_{0}^{n}$ defines a globally well-defined volume density. Considering the canonical holomorphic $(n+1,0)$-form $\Omega = dz_{1} \wedge \dots \wedge dz_{n} \wedge dw$, the fundamental identity 
\begin{center}
$\Omega \wedge \bar{\Omega} = c' ||\Omega||_{\omega_{{\rm{CY}}}}^{2} \omega_{{\rm{CY}}}^{n+1}$ 
\end{center}
holds globally for some positive constant $c'$. For the metric $\omega_{{\rm{CY}}}$ to be Ricci-flat, the norm $||\Omega||_{\omega_{{\rm{CY}}}}$ must be globally constant, forcing the conformal factor of the volume form to be a constant $C > 0$. This reduces the complex Monge-Ampère equation to the exact ODE
\begin{equation}
    U(s)^{n} U'(s) = C.
\end{equation}
Integrating both sides with respect to $s$ we obtain
\begin{equation}
    \frac{U(s)^{n+1}}{n+1} = C s + C_{0}.
\end{equation}
To ensure that the horizontal component restricts to $\omega_{0}$ over the base $X$, we impose $U(0) = 1$, implying $C_{0} = \frac{1}{n+1}$. Solving for $U(s)$ yields
\begin{equation}
    U(s) = (1 + c s)^{\frac{1}{n+1}},
\end{equation}
where $c = C(n+1) > 0$. The vertical scale factor $F'(s) = \frac{(1 + c s)^{\frac{1}{n+1}} - 1}{s}$ is smooth at $s=0$ with 
\begin{equation}
\lim_{s \to 0} F'(s) = \frac{c}{n+1} > 0, 
\end{equation}
ensuring $\omega_{{\rm{CY}}}$ extends to a smooth K\"{a}hler metric over $Z$. 

Finally, to establish completeness, we consider a $C^1$ path $\gamma: [0, a) \rightarrow Z$. If the path is horizontally divergent in the base $X$, the completeness of $(X, \omega_{0})$ ensures infinite length. If the path is not horizontally divergent, it must be vertically divergent along the radial fiber direction. In this case, the geodesic length is dictated by the integral
\begin{equation}
\int_{0}^{+\infty} \sqrt{\frac{V(s)}{s}} ds = 2\sqrt{\frac{c}{n+1}} \int_{0}^{+\infty} \frac{1}{(1+cs)^{\frac{n}{2(n+1)}}} d(s^{\frac{1}{2}}).
\end{equation}
The divergence of this integral as $s \to \infty$ implies that the path $\gamma$ has infinite length. Therefore, from Proposition \ref{completemanifold}, it follows that $(Z,\omega_{\rm{CY}})$ is complete.
\end{proof}

\begin{remark} In the setting of the above theorem, it is worth pointing out the following facts.
\begin{itemize}
    \item The projection $p:{\bf{K}}_{X}\to X$ is a homotopy equivalence (its fiber is $\mathbbm{C}$, which is contractible). Hence the pullback
    \begin{equation}
    p^{\ast}: H^{1,1}(X,\mathbbm{R}) \longrightarrow H^{1,1}({\bf{K}}_{X},\mathbbm{R})
    \end{equation}
    is an isomorphism of real vector spaces.
    \item For any class $[\eta]\in H^{1,1}({\bf{K}}_{X},\mathbbm{R})$, there exists a unique class $[\chi]\in H^{1,1}(X,\mathbbm{R})$ such that $[\eta]=p^{\ast}[\chi]$.
    \item Choose a closed representative $\chi$ of $[\chi]$ (for instance, the unique $G$-invariant representative on $X$). Then $p^{\ast}\chi$ is a closed representative of $[\eta]$.
    \item The term $-\sqrt{-1}\partial\bar\partial\phi(s)$ is exact (it equals $dd^c$ of a function). Consequently,
    \begin{equation}
    \chi_{\text{tot}} = p^{\ast}\chi - \sqrt{-1}\partial\bar\partial\phi(s)
    \end{equation}
    lies in the same cohomology class $[\eta]$.
\end{itemize}
Therefore, for every $[\eta]\in H^{1,1}({\bf{K}}_{X},\mathbbm{R})$ there exists a representative of the form $p^{\ast}\chi - \sqrt{-1}\partial\bar\partial\phi(s)$.
\end{remark}

\section{Global Existence for Tangent Phase Equations}

Consider the following standard results \cite{malley2005theory}.

\begin{theorem}[Picard-Lindel\"of]
\label{thm:picard}
Let $U \subset \mathbbm{R} \times \mathbbm{R}^{n}$ be an open set, and let $f\colon U \to \mathbbm{R}^{n}$ be continuous and locally Lipschitz in the second variable (uniformly in the first variable on compact intervals). Then for any initial condition $(s_{0}, y_{0}) \in U$, there exists a unique maximal solution $y\colon I_{\max} \to \mathbbm{R}^{n}$ of the initial value problem
\begin{equation}
\begin{cases}y'(s) = f(s, y(s)),\\
y(s_{0}) = y_{0},
\end{cases}
\end{equation}
where $I_{\max}$ is an open interval containing $s_{0}$. Moreover, if $I_{\max} \neq \mathbbm{R}$, then as $s$ approaches the endpoint of $I_{\max}$, the solution either leaves every compact subset of $U$ or blows up to infinity.
\end{theorem}

\begin{theorem}[Gronwall's Inequality]
\label{thm:gronwall}
Let $u, v$ be nonnegative, continuous functions on $[a,b]$, $C \geq 0$ be a constant, and assume that
\begin{equation}
    v(s) \leq C + \int_{a}^{s}v(\tau)u(\tau) d\tau,
\end{equation}
for $s \in [a, b]$. Then
\begin{equation}
    v(s) \leq C\exp \Big( \int_{a}^{s}u(\tau)d\tau \Big), \ \ s \in [a, b].
\end{equation}
In particular, if $C = 0$, then $v(s) \equiv 0$.
\end{theorem}

To establish the rigorous solvability of the ODE associated with the deformed Hermitian-Yang-Mills equation (Theorem \ref{TheoremA}), we isolate the underlying analytical mechanics in the general result of ordinary differential equations provided by the following lemma.

\begin{lemma}[Global Existence for Tangent Phase Equations]
\label{lemmaEDO}
Let $I = [0, \infty)$ and consider the initial value problem (IVP):
\begin{equation}
\label{IVPItangent}
    \begin{cases}
        y'(s) = v(s) \tan\big(H(s, y(s))\big), \\
        y(0) = y_{0},
    \end{cases}
\end{equation}
where $v \colon I \to \mathbbm{R}$ and $H \colon I \times \mathbbm{R} \to \mathbbm{R}$ are functions of class $C^{1}$. Suppose the system satisfies the following three conditions:
\begin{enumerate}
    \item[(A)] $H(0, y_{0}) \in (-\frac{\pi}{2}, \frac{\pi}{2})$.
    \item[(B)] For all $(s, y) \in I \times \mathbbm{R}$, we have $v(s) > 0$ and 
    \begin{equation}
        \frac{\partial H}{\partial y}(s, y) < 0.
    \end{equation}
    \item[(C)] For any $S > 0$, there exist constants $R_{S} > 0$ and $M_S > 0$ such that 
    \begin{equation}
        \big|v(s)\tan(H(s,y))\big| \le M_S(1+|y|), 
    \end{equation}
for all $s \in [0,S]$ and $|y| > R_{S}$ such that $H(s,y) \in (-\pi/2, \pi/2)$.
\end{enumerate}
Then, the IVP given in Eq. (\ref{IVPItangent}) admits a unique global solution $y \colon [0, \infty) \to \mathbbm{R}$, and the dynamical phase strictly satisfies the confinement $H(s, y(s)) \in \big(-\frac{\pi}{2}, \frac{\pi}{2} \big)$ for all $s \ge 0$.
\end{lemma}

\begin{proof}[{\bf{Proof}}]
From Condition $(A)$, observing that $H$ is continuous, let $U \subset I \times \mathbbm{R}$ be an  open neighborhood of $(0,y_{0})$ where 
\begin{center}
$H(s,y)\in(-\pi/2,\pi/2), \forall (s,y) \in U$. 
\end{center}
On $U$, consider the $C^{1}$ map $f \colon U \to \mathbbm{R}$ defined by
\begin{center}
$f(s,y):= v(s) \tan\big(H(s, y)\big)$.
\end{center}
Since $f$ is locally Lipschitz on $U$ in the second variable, by Theorem \ref{thm:picard}, there exists a unique maximal solution $y \colon [0,\tau_{\max})\to\mathbbm{R}$ for the IVP given in Eq. (\ref{IVPItangent}), with $\tau_{\max}>0$.

\medskip

Now define $u(s)=H(s,y(s))$. Since $H$ and $y$ are $C^{1}$, it follows that $u$ is $C^{1}$ and
\begin{center}
$u'(s)=H_{s}(s,y(s))+H_{y}(s,y(s))y'(s)=H_{s}(s,y(s))+H_{y}(s,y(s))v(s)\tan(u(s))$.
\end{center}
Now, from condition (B), $v(s)>0$ and $H_{y}(s,y)<0$ for all $(s,y)$, we prove that $u(s)$ never reaches $\pm\pi/2$.
\begin{enumerate}
\item[\underline{Case 1.}] $u < \pi/2$: Suppose that there exists $s_{0}>0$ such that $u(s_{0})=\pi/2$. Because $u$ is continuous and $u(0)=H(0,y_{0})\in(-\pi/2,\pi/2)$, we can define the \textit{first} time
\begin{equation}
s_{0}=\min\{s>0 \ | \ u(s)=\pi/2\},\ \  u(s)<\pi/2\;\forall s\in[0,s_{0}).
\end{equation}

For $s$ close to $s_{0}$ from below, $u(s)\in(\pi/2-\delta,\pi/2)$ with $\delta>0$ small. On this interval $\tan(u)\to+\infty$ as $u\to\pi/2^-$. Moreover, $H_{y}<0$ and $v>0$, so the product $H_{y} v(s)\tan(u(s))$ is negative and its norm tends to $+\infty$.\\

Since $H_{s}$ is continuous, it is bounded on the compact set $\{(s,y(s)) \ | \ s\in[0,s_{0}]\}$. Hence, there exists a constant $C>0$ such that 
\begin{center}
$|H_{s}(s,y(s))|\le C$, \ \ \ $\forall s\in[0,s_{0}]$.
\end{center}
Choose $\delta>0$ sufficiently small so that whenever $u\in(\pi/2-\delta,\pi/2)$,
\begin{equation}
H_{y}(s,y(s)) v(s)\tan(u(s))\le -(C+2).
\end{equation}
This is possible because the left‑hand side tends to $-\infty$ as $u\to\pi/2^-$. Consequently, for all $s$ in a left neighborhood $(s_{0}-\varepsilon,s_{0})$, with $\varepsilon >0$ sufficiently small, we have
\begin{center}
$u'(s)=H_{s}+H_{y} v(s)\tan(u)\le C-(C+2)=-2\le -1$.
\end{center}

Now integrating the inequality $u'(w)\le -1$ from $s$ to $s_{0}$ (with $s_{0}-\varepsilon<s<s_{0}$), we obtain
\begin{equation}
u(s_{0})-u(s)=\int_{s}^{s_{0}}u'(w)dw\le\int_s^{s_{0}}(-1)dw=-(s_{0}-s).
\end{equation}
Thus
\begin{center}
$u(s)\ge u(s_{0})+(s_{0}-s)=\frac{\pi}{2}+(s_{0}-s)>\frac{\pi}{2}$,
\end{center}
which contradicts the fact that $u(s)<\pi/2$ for all $s<s_{0}$. Therefore, we conclude $u(s)<\pi/2$ for every $s\ge0$.

\item[\underline{Case 2.}] $u > - \pi/2$: The argument is symmetric. Suppose there exists a first time $s_{1}>0$ with 
\begin{center}
$u(s_{1})=-\pi/2$, \ \ \ $u(s)>-\pi/2$ for $s\in[0,s_{1})$. 
\end{center}
On this interval, $\tan(u)\to-\infty$ as $u\to-\pi/2^+$. Because $H_{y}<0$ and $v>0$, the term $H_{y} v(s)\tan(u(s))$ tends to $+\infty$. As in the previous case, we have 
\begin{center}
$|H_{s}(s,y(s))|\le C$ 
\end{center}
for some constant $C > 0$ on the interval $[0,s_{1}]$. Choose $\delta>0$ sufficiently small so that $u\in(-\pi/2,-\pi/2+\delta)$ and
\begin{center}
$H_{y}(s,y(s))v(s)\tan(u(s))\ge C+2$,
\end{center}
which yields $u'(s)=H_{s}+H_{y} v(s)\tan(u)\geq -C+(C+2)=2\ge 1$ on a left neighborhood of $s_{1}$. Integrating from $s$ to $s_{1}$ gives
\begin{equation}
u(s_{1})-u(s)=\int_{s}^{s_{1}}u'(w)dw\geq\int_{s}^{s_{1}}1dw=s_{1}-s,
\end{equation}
hence
\begin{center}
$u(s)\le u(s_{1})-(s_{1}-s)=-\frac{\pi}{2}-(s_{1}-s)<-\frac{\pi}{2}$,
\end{center}
contradicting $u(s)>-\pi/2$ for $s<s_{1}$. Thus $u(s)>-\pi/2$ for all $s\geq 0$.
\end{enumerate}
Combining Case 1 and Case 2, we conclude that 
\begin{equation}
-\frac{\pi}{2}<H(s,y(s))=u(s)<\frac{\pi}{2}, \ \ \ \forall s\geq 0,
\end{equation}
which completes the proof of phase confinement.

\medskip

Fix an arbitrary $S>0$. By the phase confinement established in the previous step, $u(s) = H(s,y(s))$ remains strictly inside $(-\pi/2, \pi/2)$. Thus, on the bounded region $|y(s)| \le R_{S}$, the continuous function $v(s)\tan(H(s,y(s)))$ is uniformly bounded by some constant $K_{S}$. For the unbounded region $|y(s)| > R$, condition (C) provides the linear bound $M_{S}(1+|y(s)|)$. Combining these regimes, there exists a global constant $\tilde{M}_{S} = \max(K_{S}, M_{S}) > 0$ such that
\begin{center}
$|y'(s)| \le \tilde{M}_S(1+|y(s)|) \ \ \forall s\in[0,S]\cap[0,\tau_{\max})$.
\end{center}
For $s$ in this interval, set $z(s)=1+|y(s)|$. Then, we obtain 
\begin{center}
$z'(s) = {\rm{sgn}}(y(s))y'(s) \leq |y'(s)| \leq \tilde{M}_{S} (1 + |y(s)|) = \tilde{M}_{S}z(s)$
\end{center}
From above, it follows that 
\begin{equation}
z(s) \leq z(0) + \int_{0}^{s}\tilde{M}_{S}z(u)du.
\end{equation}
Therefore, Theorem \ref{thm:gronwall} yields
\begin{equation}
y(s) < z(s) \leq z(0)\exp \Big ( \int_{0}^{s}\tilde{M}_{S}du\Big) = z(0)e^{\tilde{M}_{S}s} \leq (1+|y_{0}|) e^{\tilde{M}_{S} S},
\end{equation}
for every $s\in[0,\min\{S,\tau_{\max}\})$. Thus $y(s)$ remains bounded on $[0,\min\{S,\tau_{\max}\})$. If $\tau_{\max}\le S$, then the solution would have a finite limit as $s\to\tau_{\max}^{-}$, and by standard continuation theory it could be extended beyond $\tau_{\max}$, contradicting maximality. Hence $\tau_{\max}>S$. Since $S>0$ is arbitrary, it follows that $\tau_{\max}=+\infty$. Hence, we conclude that the IVP given in Eq. (\ref{IVPItangent}) admits a unique global solution $y \colon [0, \infty) \to \mathbbm{R}$.

\medskip

As we see from above, the solution $y$ exists globally, is unique, and satisfies the phase confinement property. This completes the proof.
\end{proof}

\section{Proof of Main Results}

Now we can prove Theorem \ref{TheoremA}. In what follows, $X = G^{\mathbbm{C}}/P$ is a rational homogeneous variety defined by some parabolic Lie subgroup $P = P_{I} \subset G^{\mathbbm{C}}$. Also, we denote by $\omega_{0}$ the unique $G$-invariant K\"{a}hler-Einstein metric on $X$ satisfying ${\rm{Ric}}(\omega_{0}) = \omega_{0}$.

\begin{proof}[{\bf{Proof of Theorem \ref{TheoremA}}}]
Given $[\eta] = p^{\ast}[\chi] \in H^{1,1}({\bf{K}}_{X}, \mathbbm{R})$, we consider the global $(1,1)$-form ansatz using a radial smooth potential as follows
\begin{equation}
    \Upsilon_{\phi} = p^{\ast}(\chi) - \sqrt{-1} \partial \bar{\partial} \phi(s) \in [\eta].
\end{equation}
Here we take $\chi$ as the unique $G$-invariant solution to the dHYM equation on the base $X$ obtained from Theorem \ref{theoremDHYM}.

Computing the derivative term with the chain rule for $s = r^{2}$ we obtain
\begin{equation}
    \sqrt{-1} \partial \bar{\partial} \phi(s) = \phi'(s) \sqrt{-1} \partial \bar{\partial} s + \phi''(s) \sqrt{-1} \partial s \wedge \bar{\partial} s.
\end{equation}
Recalling that $\sqrt{-1} \partial \bar{\partial} s = s p^{\ast}(\omega_{0}) + \frac{\sqrt{-1} \partial s \wedge \bar{\partial} s}{s}$, we substitute and organize the coefficients as
\begin{equation}
    -\sqrt{-1} \partial \bar{\partial} \phi(s) = -s\phi'(s) p^{\ast}\omega_{0} - (\phi'(s) + s\phi''(s)) \frac{\sqrt{-1} \partial s \wedge \bar{\partial} s}{s}.
\end{equation}
In order to linearize the fiber geometry, we define the radial momentum evaluated at the function $s$ as 
\begin{equation}
\psi(s) := -s\phi'(s). 
\end{equation}
Thus, we have $\psi'(s) = -(\phi'(s) + s\phi''(s))$. From this, the representative form on the total space can be written in the following way
\begin{equation}
    \Upsilon_{\phi} = p^{\ast}(\chi) + \psi(s)p^{\ast}(\omega_{0}) + \psi'(s) \frac{\sqrt{-1} \partial s \wedge \bar{\partial} s}{s}.
\end{equation}

We must now find the eigenvalues of the endomorphism $\omega_{{\rm{CY}}}^{-1} \Upsilon_{\phi}$. Since 
\begin{center}
$\displaystyle \omega_{{\rm{CY}}}= U(s) p^{\ast}\omega_{0} + V(s) \frac{\sqrt{-1} \partial s \wedge \bar{\partial} s}{s}$ 
\end{center}
and the horizontal and vertical spaces are orthogonal with respect to both forms. Therefore, the endomorphism $\omega_{{\rm{CY}}}^{-1} \Upsilon_{\phi}$ diagonalizes in blocks. In the horizontal direction, we have the following description 
\begin{center}
$U(s)^{-1} \omega_{0}^{-1} (\chi + \psi(s)\omega_{0})$. 
\end{center}
Since $\chi$ is the invariant representative, the matrix $\omega_{0}^{-1} \circ \chi$ has constant eigenvalues ${\bf{q}}_{\beta}(\omega_{0}^{-1} \circ \chi)$ for each positive root $\beta \in \Phi_{I}^{+}$, see for instance Proposition \ref{eigenvalueatorigin}. The horizontal eigenvalues of $\omega_{{\rm{CY}}}^{-1} \circ \Upsilon_{\phi}$ are therefore
\begin{equation}
    \Lambda_{\beta}(s) := \frac{{\bf{q}}_{\beta}(\omega_{0}^{-1} \circ \chi) + \psi(s)}{U(s)}, \ \ \ \ \beta \in \Phi_{I}^{+}.
\end{equation}
The vertical block is one-dimensional, generating the isolated fiber eigenvalue 
\begin{equation}
    \Lambda_{{\rm{fib}}}(s) = \frac{\psi'(s)}{V(s)}.
\end{equation}
Since the dHYM equation requires that the sum of the arctangents of the eigenvalues equals the constant phase $\hat{\Theta}_{{\rm{tot}}}$, we obtain the following equation
\begin{equation}
    \sum_{\beta \in \Phi_{I}^{+}} \arctan\Bigg ( \frac{{\bf{q}}_{\beta}(\omega_{0}^{-1} \circ \chi) + \psi(s)}{U(s)} \Bigg) + \arctan\Bigg ( \frac{\psi'(s)}{V(s)} \Bigg) = \hat{\Theta}_{{\rm{tot}}}.
\end{equation}
Isolating the vertical term and applying the tangent function to both sides, we obtain the fundamental ODE for the scalar variable $s$ as follows
\begin{equation}
    \frac{d}{ds}\psi(s) = V(s) \tan\Bigg ( \hat{\Theta}_{{\rm{tot}}} - \sum_{\beta \in \Phi_{I}^{+}} \arctan\Bigg ( \frac{{\bf{q}}_{\beta}(\omega_{0}^{-1} \circ \chi) + \psi(s)}{U(s)} \Bigg) \Bigg ),
\end{equation}
which concludes the proof of item (1). 

\medskip

We now cast this ODE into the framework of Lemma \ref{lemmaEDO} by identifying $y(s) = \psi(s)$, $v(s) = V(s)$, and the dynamical phase function as
\begin{equation}
    H(s, y) := \hat{\Theta}_{{\rm{tot}}} - \sum_{\beta \in \Phi_{I}^{+}} \arctan\Bigg ( \frac{{\bf{q}}_{\beta}(\omega_{0}^{-1} \circ \chi) + y}{U(s)} \Bigg).
\end{equation}
We now verify the necessary conditions (A), (B) and (C) for global existence:
\begin{itemize}
    \item \underline{Condition (A)}: Suppose that the phase calibration $\hat{\Theta}_{{\rm{tot}}}$ satisfies 
        \begin{equation}
        \hat{\Theta}_{{\rm{tot}}}  \in \big (\Theta_{\omega_{0}}(\chi) -\frac{\pi}{2}, \Theta_{\omega_{0}}(\chi) + \frac{\pi}{2}\big), 
    \end{equation}
At the zero section ($s=0$), we require a regular form, thus $\psi(0) = 0$. Using the Calabi boundary condition $U(0) = 1$, the initial phase is 
\begin{equation}
H(0, 0) = \hat{\Theta}_{{\rm{tot}}} - \sum_{\beta} \arctan({\bf{q}}_{\beta}(\omega_{0}^{-1} \circ \chi)) = \hat{\Theta}_{{\rm{tot}}} - \Theta_{\omega_{0}}(\chi). 
\end{equation}
Therefore, we have $H(0, 0) \in (-\pi/2, \pi/2)$. Notice that, from this, there exists a unique maximal solution $\psi \colon [0,\tau_{\max})\to\mathbbm{R}$ for the IVP given by
\begin{equation}
\displaystyle \frac{d}{ds}\psi(s) = V(s)\tan(H(s,\psi(s))), \ \ \ \ \psi(0) = 0.
\end{equation}
From above, in order to show that $\tau_{\max} = +\infty$, we need to verify that the conditions (B) and (C) of Lemma \ref{lemmaEDO} hold.

\item \underline{Condition (B)}: Computing the partial derivative of the $H$ with respect to $y$, it follows that 
\begin{equation}
        \frac{\partial H}{\partial y}(s, y) = - \sum_{\beta \in \Phi_{I}^{+}} \frac{U(s)^{-1}}{1 + \big ( \frac{{\bf{q}}_{\beta}(\omega_{0}^{-1} \circ \chi) + y}{U(s)} \big)^{2}}.
\end{equation}
Since $U(s) \ge 1 > 0$ for all $s \ge 0$, this sum consists of strictly negative terms. Thus, we obtain that 
\begin{center}
$\displaystyle \frac{\partial H}{\partial y}(s,y) < 0$. 
\end{center}
Furthermore, $V(s) = U'(s) > 0$. Notice that this provides the exact geometric stabilization required to repel the phase from the singular asymptotes, that is, from the phase confinement argument presented in the proof of Lemma \ref{lemmaEDO}, it follows that 
\begin{center}
$u(s) = H(s,\psi(s)) \in (-\pi/2,\pi/2)$
\end{center}
for every $s \in [0,\tau_{\max})$.

\item \underline{Condition (C)}: Given $S > 0$, we need to show that there exist constants $R_{S} > 0$ and $M_S > 0$ such that 
    \begin{equation}
        \big|V(s)\tan(H(s,y))\big| \le M_S(1+|y|), 
    \end{equation}
for all $s \in [0,S]$ and $|y| > R_{S}$ such that $H(s,y) \in (-\pi/2, \pi/2)$. In order to prove this, we proceed as follows. Using the identity 
\begin{center}
$\displaystyle \arctan(x) = \text{sgn}(x)\frac{\pi}{2} - \frac{1}{x} + \mathcal{O}(x^{-3}),$ 
\end{center}
for large $|x|$, we can expand the terms inside the phase for $|y| \gg 1$ in the following way
\begin{equation}
        \arctan\Bigg ( \frac{{\bf{q}}_{\beta}(\omega_{0}^{-1} \circ \chi) + y}{U(s)} \Bigg) = \text{sgn}(y) \frac{\pi}{2} - \frac{U(s)}{y} + \mathcal{O}\bigg (\frac{1}{y^{2}}\bigg).
    \end{equation}
Note that $\operatorname{sgn}(y)$ determines the sign of the limit $\pm\pi/2$, and the sum over $\beta\in\Phi_I^{+}$ contains $n = |\Phi_I^{+}|$ terms, recall that $\dim_{\mathbbm{C}}(X) = n$. 

If we have $L \in (-\pi/2, \pi/2)$ such that 
\begin{equation}
\lim_{|y| \to \infty} H(s, y) = L,
\end{equation}
there exists a radius $R_{S} > 0$ such that $|\tan(H(s,y))|$ is uniformly bounded by some constant $K > 0$ for all $|y| > R_{S}$. Because the geometric scale factor $V(s)$ is continuous on the compact interval $[0, S]$, by choosing $\displaystyle M_{S} = K  \max_{s \in [0,S]} V(s)$, it follows that 
 \begin{center}
    $\big|V(s)\tan(H(s,y))\big| \le M_S(1+|y|)$, 
\end{center}
for all $s \in [0,S]$ and $|y| > R_{S}$ such that $H(s,y) \in (-\pi/2, \pi/2)$.
    
The critical scenario occurs when the topological data forces the limit to be exactly $\pm \pi/2$. Observing that
\begin{equation}
 \lim_{y \to +\infty} H(s, y) = \hat{\Theta}_{{\rm{tot}}} - n\frac{\pi}{2} \ \ \ \text{and} \ \ \ \lim_{y \to -\infty} H(s, y) = \hat{\Theta}_{{\rm{tot}}} + n\frac{\pi}{2},
\end{equation}
the critical scenario occurs when the topological data forces 
\begin{equation}
\hat{\Theta}_{{\rm{tot}}} - n \kappa \frac{\pi}{2} = \sigma \frac{\pi}{2}, \ \ \  \kappa, \sigma \in \{+1, -1\}.
\end{equation}
In this critical scenario, the phase behaves as
\begin{equation}
        H(s, y) = \pm \frac{\pi}{2} + n \frac{U(s)}{y} + \mathcal{O}\bigg (\frac{1}{y^{2}}\bigg).
\end{equation}
Utilizing the Laurent expansion of the tangent function near its singularities, i.e.,
\begin{center}
$\displaystyle \tan\!\left(\pm \frac{\pi}{2} + \varepsilon\right) = -\frac{1}{\varepsilon} + \mathcal{O}(\varepsilon)$, 
\end{center}
we deduce the dominant growth as follows. Considering 
\begin{center}
$\displaystyle \varepsilon = n \ \frac{U(s)}{y} + \mathcal{O}\bigg (\frac{1}{y^{2}}\bigg),$
\end{center}
it follows that 
\begin{center}
$\displaystyle \frac{1}{\varepsilon} = \frac{1}{nU(s)/y + \mathcal{O}(1/y^{2})} = \frac{y}{nU(s)} \cdot \frac{1}{1 + \mathcal{O}(1/y)} = \frac{y}{nU(s)} + \mathcal{O}(1),$
\end{center}
where $\mathcal{O}(1)$ denotes term that remains bounded as $|y| \to \infty$. Here we have used that
\begin{equation}
\frac{\mathcal{O}(1/y^{2})}{nU(s)/y} = \mathcal{O}(1/y) \ \ \ \text{and} \ \ \ \frac{1}{1 + \mathcal{O}(1/y)} = 1 - \mathcal{O}(1/y) +  \mathcal{O}(1/y^{2}) - \cdots.
\end{equation}
     
     From above, we obtain
    \begin{equation}
        \big|\tan(H(s, y))\big| = \left| -\frac{y}{n \cdot U(s)} + \mathcal{O}(1) \right| \le C \frac{|y|}{U(s)},
    \end{equation}
    for some constant $C>0$ when $|y|$ is sufficiently large. Multiplying by the fiber scale factor $V(s)$, the dominant term is proportional to $\frac{V(s)}{U(s)}|y|$. Since $U(s)$ is strictly bounded away from zero and $V(s)$ is continuous on the compact interval $[0, S]$, the ratio is bounded. Thus, there exists a constant $M_S > 0$ and a constant $R_{S} \gg 0$ such that 
    \begin{center}
    $|V(s)\tan(H(s,y))| \leq M_S(1 + |y|)$. 
    \end{center}
for all $s \in [0,S]$ and $|y| > R_{S}$ such that $H(s,y) \in (-\pi/2, \pi/2)$.
\end{itemize}

From the conditions (A), (B) and (C) verified above, it follows from Lemma \ref{lemmaEDO} that the initial value problem
\begin{equation}
 \frac{d}{ds}\psi(s) = V(s) \tan (H(s,\psi(s))), \ \ \ \psi(0) = 0,
\end{equation}
admits a unique global solution $\psi \colon [0, \infty) \to \mathbbm{R}$, and the dynamical phase strictly satisfies the confinement $H(s, \psi(s)) \in \big(-\frac{\pi}{2}, \frac{\pi}{2} \big)$ for all $s \ge 0$.
\medskip

Now since 
\begin{center}
$\psi'(0)=V(0)\tan\big(\hat{\Theta}_{{\rm{tot}}}-\Theta_{\omega_{0}}(\chi)\big),$ 
\end{center}
is finite because the argument is in $(-\pi/2,\pi/2)$ by the phase condition, it follows that 
\begin{equation}
\psi(u)=\psi'(0)u+\mathcal{O}(u^{2}), 
\end{equation}
near $0$. From above, it follows that $\psi(u)/u$ is smooth and the integral 
\begin{equation}
\phi(s) := -\int_{0}^{s}\frac{\psi(u)}{u}du,
\end{equation}
yields a $C^{\infty}$ function $\phi(s)$ on $[0,\infty)$. From above, we have that $\Upsilon_{\phi} \in [\eta]$ defines a smooth solution to the dHYM equation on $Z = {\rm{Tot}}({\bf{K}}_{X})$, concluding the proof.
\end{proof}

\begin{proof}[{\bf{Proof of Corollary \ref{CorollaryA}}}]
Let $h$ be a smooth Hermitian metric on the holomorphic line bundle ${\bf{L}} \to X_{P}$ such that its Chern curvature form satisfies $\sqrt{-1}F_{{\bf{h}}} = \chi$, where $\chi$ is the unique $G$-invariant representative of $c_{1}({\bf{L}})$. From this, we construct a radially symmetric Hermitian metric on the pullback line bundle $\mathcal{L} = p^{\ast}{\bf{L}} \to Z$ by setting
\begin{equation}
\label{radansatz}
    h := p^{\ast}{\bf{h}} \cdot e^{\phi(s)},
\end{equation}
where $\phi(s)$ is a smooth radial potential depending strictly on the squared norm coordinate $s$ of the fibers. 

The curvature of the Chern connection associated with $h$ is given by $F_{h} = p^{\ast}F_{{\bf{h}}} + \bar{\partial}\partial \phi(s)$. Substituting $F_{{\bf{h}}} = -\sqrt{-1}\chi$ and recalling that $\bar{\partial}\partial = -\partial\bar{\partial}$, we obtain the following expression
\begin{equation}
    F_{h} = -\sqrt{-1}p^{\ast}\chi - \partial\bar{\partial}\phi(s) = -\sqrt{-1} \Big( p^{\ast}\chi - \sqrt{-1}\partial\bar{\partial}\phi(s) \Big) = -\sqrt{-1}\Upsilon_{\phi},
\end{equation}
where $\Upsilon_{\phi}$ is the global $(1,1)$-form defined as in Theorem \ref{TheoremA}. Substituting $-F_{h} = \sqrt{-1}\Upsilon_{\phi}$ into the dHYM equation for $\mathcal{L}$, the PDE becomes
\begin{equation}
\label{dHYM_Bundle}
    {\rm{Im}}\Big (\omega_{{\rm{CY}}} + \sqrt{-1}\Upsilon_{\phi} \Big )^{n+1} = \tan(\hat{\Theta}_{{\rm{tot}}}) {\rm{Re}}\Big (\omega_{{\rm{CY}}} + \sqrt{-1}\Upsilon_{\phi} \Big )^{n+1}.
\end{equation}
Since the Lagrangian phase of $\chi$ on the base satisfies $\hat{\Theta}({\bf{L}}) = \Theta_{\omega_{0}}(\chi) \ ({\rm{mod}} \ 2\pi)$, the hypothesis on the phase $\hat{\Theta}_{{\rm{tot}}} \in \big(\hat{\Theta}({\bf{L}}) - \frac{\pi}{2}, \hat{\Theta}({\bf{L}}) + \frac{\pi}{2}\big)$ guarantees that
\begin{equation}
\widetilde{\Theta_{{\rm{tot}}}} := \hat{\Theta}_{{\rm{tot}}} + 2\pi m \in \big(\Theta_{\omega_{0}}(\chi) - \frac{\pi}{2}, \Theta_{\omega_{0}}(\chi) + \frac{\pi}{2}\big),
\end{equation}
for some $m \in \mathbbm{Z}$. Since $\tan(\widetilde{\Theta_{{\rm{tot}}}}) = \tan(\hat{\Theta}_{{\rm{tot}}})$, by Theorem \ref{TheoremA}, there exists a unique smooth global solution $\psi(s) = -s\phi'(s)$ to the radial ODE associated to Eq. (\ref{dHYM_Bundle}). This ensures that $h$ is a globally defined smooth metric solving the dHYM equation on $\mathcal{L}$.

To prove the second part of Corollary \ref{CorollaryA}, assume that conditions (i) and (ii) hold. Substituting these conditions into the fundamental ODE from Theorem \ref{TheoremA}, we obtain the following
\begin{equation}
    \frac{d\psi}{ds} = V(s)\tan\bigg( (n+1)\arctan(q) - n\arctan\Big(\frac{q + \psi(s)}{U(s)}\Big) \bigg).
\end{equation}
We claim that the exact solution to this initial value problem is the linear form 
\begin{center}
$\psi(s) = q(U(s) - 1)$. 
\end{center}
To verify this, we observe that 
\begin{equation}
    \frac{q + \psi(s)}{U(s)} = \frac{q + q(U(s) - 1)}{U(s)} = \frac{qU(s)}{U(s)} = q.
\end{equation}
Consequently, the right-hand side of the ODE simplifies in the following form
\begin{equation}
    V(s)\tan\Big( (n+1)\arctan(q) - n\arctan(q) \Big) = V(s)\tan\big(\arctan(q)\big) = qV(s).
\end{equation}
On the left-hand side, differentiating our ansatz yields $\frac{d\psi}{ds} = qU'(s)$. Since the Calabi ansatz guarantees that $V(s) = U'(s)$, we have $\frac{d\psi}{ds} = qV(s)$. Furthermore, the initial condition $\psi(0) = q(U(0) - 1) = 0$ is trivially satisfied since $U(0) = 1$. By the uniqueness of solutions established in Lemma \ref{lemmaEDO}, we conclude that $\psi(s) = q(U(s) - 1)$ defines the desired solution.

Consequently, all horizontal and vertical eigenvalues of the endomorphism $\omega_{{\rm{CY}}}^{-1} \circ \Upsilon_{\phi}$ are identical and globally constant, equal to $q$. The trace of $\Upsilon_{\phi}$ with respect to the Calabi-Yau metric $\omega_{{\rm{CY}}}$ is given by
\begin{equation}
    \Lambda_{\omega_{{\rm{CY}}}}(\Upsilon_{\phi}) = \sum_{\beta \in \Phi_{I}^{+}} \Lambda_{\beta}(s) + \Lambda_{{\rm{fib}}}(s) = nq + q = (n+1)q.
\end{equation}
Using the curvature identification $F_{h} = -\sqrt{-1}\Upsilon_{\phi}$, we apply the contraction operator to explicitly obtain the scalar curvature of the connection:
\begin{equation}
    \sqrt{-1}\Lambda_{\omega_{{\rm{CY}}}}(F_{h}) = \sqrt{-1}\Lambda_{\omega_{{\rm{CY}}}}\big(-\sqrt{-1}\Upsilon_{\phi}\big) = \Lambda_{\omega_{{\rm{CY}}}}(\Upsilon_{\phi}) = (n+1)q.
\end{equation}
Since this trace is a strict global constant, the metric $h$ simultaneously solves the Hermitian Yang-Mills equation, which completes the proof.
\end{proof}

\section{Final Comments and Examples}

In this final section, we make some remarks concerning our main results and provide some explicit computations on the Eguchi–Hanson space and on the Calabi-Yau threefold ${\bf{K}}_{\mathbbm{P}^{2}}$.

\subsection*{Lie-theoretic Explicit Formulation} 

Following the background on rational homogeneous variety established in Section \ref{generalities}, let $X_{P} = G^{\mathbbm{C}}/P$ be a complex flag variety defined by a standard parabolic subgroup $P = P_I$ associated with a subset of simple roots $I \subset \Delta$. For each simple root $\alpha \in \Delta \backslash I$, let $\varpi_{\alpha}$ be the corresponding fundamental weight and $v_{\varpi_{\alpha}}^{+}$ be the unique highest weight vector of the finite-dimensional irreducible $G^{\mathbbm{C}}$-module $V(\varpi_{\alpha})$. Choosing a local trivializing section $s_i: U_i \to G^{\mathbbm{C}}$ of the principal $P$-bundle over an open neighborhood $U_i \subset X_{P}$, the natural $G$-invariant Hermitian structure on the fundamental line bundle $\mathscr{O}_{\alpha}(1)$ is locally given by the smooth function 
\begin{equation}
q_i(x) = \left\|\Pi_{\varpi_{\alpha}}(s_i(x))v_{\varpi_{\alpha}}^{+}\right\|^{-2}, 
\end{equation}
where $\left\|\cdot\right\|$ is the norm induced by a fixed $G$-invariant inner product on $V(\varpi_{\alpha})$.

As we have seen, the canonical bundle $K_{X_{P}}$ corresponds to the character determined by the sum of positive roots $\delta_{P} = \sum_{\beta \in \Phi_I^{+}} \beta$, which also decomposes as $\delta_{P} = \sum_{\alpha \in \Delta \backslash I} \langle \delta_{P}, \alpha^{\vee} \rangle \varpi_{\alpha}$. This induces a canonical characterization as
\begin{equation}
{\bf{K}}_{X_{P}} \cong \bigotimes_{\alpha \in \Delta \backslash I} \mathscr{O}_{\alpha}(-1)^{\otimes \langle \delta_{P}, \alpha^{\vee} \rangle}.
\end{equation}
Consequently, the fiberwise Hermitian metric ${\bf{h}}$ induced on ${\bf{K}}_{X_{P}}$ by the $G$-invariant Kähler-Einstein metric $\omega_{0}$ of the base space can be written explicitly in terms of the group representation data by 
\begin{equation}
{\bf{h}} = \prod_{\alpha \in \Delta \backslash I} \left\|\Pi_{\varpi_{\alpha}}(s_i(x))v_{\varpi_{\alpha}}^{+}\right\|^{2\langle \delta_{P}, \alpha^{\vee} \rangle}w\overline{w}, \ \  \forall x \in U_i. 
\end{equation}
By defining the local complex coordinates of the total space ${\rm{Tot}}({\bf{K}}_{X_{P}})$ as $(x, w) \in U_i \times \mathbbm{C}$, where $w$ parameterizes the vertical direction along the fiber, the squared norm coordinate $s$ of the Calabi ansatz becomes explicitly as
\begin{equation}
s(x, w) = |w|^{2} \prod_{\alpha \in \Delta \backslash I} \left\|\Pi_{\varpi_{\alpha}}(s_i(x))v_{\varpi_{\alpha}}^{+}\right\|^{2\langle \delta_{P},\alpha^{\vee} \rangle}.
\end{equation}
Substituting this algebraic relation into the exact reduction provided by Theorem \ref{TheoremA}, the global deformed Hermitian-Yang-Mills potential $\phi(x,w)$ on the local Calabi-Yau geometry is retrieved explicitly as a function of the local coordinates by
\begin{equation}
\phi(x, w) = - \int_{0}^{\displaystyle |w|^{2} \prod_{\alpha \in \Delta \backslash I} \left\|\Pi_{\varpi_{\alpha}}(s_i(x))v_{\varpi_{\alpha}}^{+}\right\|^{2\langle \delta_{P}, \alpha^{\vee} \rangle}} \frac{\psi(u)}{u} du,
\end{equation}
where $\psi(u)$ is the unique smooth solution to the radial ordinary differential equation (\ref{dHYMODE}). This formulation unifies the analysis of fully nonlinear PDEs on local Calabi-Yau manifolds with the rigid algebraic symmetries embedded in the root data of rational homogeneous spaces.

For a holomorphic line bundle ${\bf{L}} \in \operatorname{Pic}(X)$ with associated weight 
\begin{equation}
\lambda({\bf{L}}) = \sum_{\alpha \in \Delta \backslash I} \langle c_{1}({\bf{L}}), [\mathbbm{P}_{\alpha}^{1}] \rangle  \varpi_{\alpha},
\end{equation}
the unique $G$-invariant Hermitian metric ${\bf{h}}$ on ${\bf{L}}$ whose curvature is the $G$-invariant representative $\chi \in c_{1}({\bf{L}})$ is given locally by
\begin{equation}
{\bf{h}} = \prod_{\alpha \in \Delta \backslash I} \left\| \Pi_{\varpi_{\alpha}}(s_{i}(x)) v_{\varpi_{\alpha}}^{+} \right\|^{-2 \langle \lambda({\bf{L}}), \alpha^{\vee} \rangle}a\overline{a}.
\end{equation}
Indeed, a direct computation using the identities from Section \ref{generalities} gives
\begin{equation}
\sqrt{-1}F_{{\bf{h}}} = \sum_{\alpha \in \Delta \backslash I} \langle \lambda({\bf{L}}), \alpha^{\vee} \rangle  {\bf{\Omega}}_{\alpha} = \chi,
\end{equation}
where ${\bf{\Omega}}_{\alpha}$ are the $G$-invariant $(1,1)$-forms associated to $\mathscr{O}_{\alpha}(1)$. Substituting this explicit expression for ${\bf{h}}$ into the radial ansatz (\ref{radansatz}) yields the following explicit formula for the Hermitian metric on $\mathcal{L} = p^{\ast}{\bf{L}}$ 
\begin{equation}
h = \prod_{\alpha \in \Delta \backslash I} \left\| \Pi_{\varpi_{\alpha}}(s_{i}(x)) v_{\varpi_{\alpha}}^{+} \right\|^{-2 \langle \lambda({\bf{L}}), \alpha^{\vee} \rangle}\exp \Bigg ( -\int_{0}^{ |w|^{2} \prod_{\alpha \in \Delta \backslash I} \left\|\Pi_{\varpi_{\alpha}}(s_i(x))v_{\varpi_{\alpha}}^{+}\right\|^{2\langle \delta_{P}, \alpha^{\vee} \rangle}} \frac{\psi(u)}{u} du\Bigg)a \overline{a},
\end{equation}
that solves the dHYM equation
\begin{equation}
 {\rm{Im}}\big (\omega_{{\rm{CY}}} - F_{h} \big )^{n+1} = \tan(\hat{\Theta}_{{\rm{tot}}}) {\rm{Re}}\big (\omega_{{\rm{CY}}} -F_{h} \big )^{n+1},
\end{equation}
with  $\hat{\Theta}_{{\rm{tot}}} \in \big ( \hat{\Theta}({\bf{L}}) - \frac{\pi}{2}, \hat{\Theta}({\bf{L}}) + \frac{\pi}{2} \big)$. From above, we have a robust constructive mechanism to obtain explicit solutions to the dHYM equation on non-compact Calabi-Yau manifolds.

\subsection*{Generalizations via Momentum Profile Construction} 

In \cite{HwangSinger2002}, Hwang–Singer's momentum construction combines Calabi's ansatz with ideas from symplectic geometry to produce complete Kähler metrics on circle-invariant subbundles of a Hermitian holomorphic line bundle $p:(L,h)\to (M,\omega_{M})$. The construction is governed by a momentum profile $\varphi$, which encodes the metric in momentum coordinates. Applying this to the total space ${\rm{Tot}}({\bf{E}})$ of every negative line bundle of the form 
\begin{equation}
{\bf{E}} = \frac{m}{I(X_{P})}{\bf{K}}_{X_{P}} \in {\rm{Pic}}(X_{P}),
\end{equation}
such that $m \in \mathbbm{Z}_{+}$ and $I(X_{P})$ denotes the Fano index of $X_{P}$, we obtain a complete Kähler metric of the form 
\begin{equation}
\omega_{\varphi} = U_{\varphi}(s) p^{\ast}\omega_{0} + V_{\varphi}(s) \frac{\sqrt{-1} \partial s \wedge \bar{\partial} s}{s},
\end{equation}
where $s$ denotes the squared norm on the fibers, $\varphi \colon [0,+\infty) \to (0,+\infty)$ is a smooth function called the momentum profile of $\omega_{\varphi}$, and
\begin{equation}
U_{\varphi}(s) = 1 + \frac{m}{I(X_{P})}\tau(s), \qquad V_{\varphi}(s) = \varphi(\tau(s)),
\end{equation}
such that $\tau(s)$ is determined by the differential relation
\begin{equation}
\frac{ds}{s} = \frac{d \tau}{\varphi(\tau)}.
\end{equation}
Here we also have $V_{\varphi} = \frac{I(X_{P})}{m}U_{\varphi}'$, together with $U_{\varphi}(s) > 0$ and $V_{\varphi}(s) > 0$ for all $s \geq 0$. In particular, when $m = I(X_{P})$, we recover the Calabi–Yau case ${\bf{E}} = {\bf{K}}_{X_{P}}$ with $V = U'$, as considered in Theorem \ref{TheoremA}. As it can be seen, the ODE reduction and the global existence argument presented in this work for $({\rm{Tot}}({\bf{K}}_{X_{P}}),\omega_{{\rm{CY}}})$ can be naturally further explored in the case $({\rm{Tot}}({\bf{E}}),\omega_{\varphi})$. In this general setting the ODE to be approached has the following general form
\begin{equation}
\label{generaldHYMprofile}
 \frac{d}{ds}\psi(s) = V_{\varphi}(s) \tan\Bigg ( \displaystyle \hat{\Theta}_{{\rm{tot}}} - \sum_{\beta \in \Phi_{I}^{+}} \arctan\Bigg( \frac{{\bf{q}}_{\beta}(\omega_{0}^{-1} \circ \chi) + \psi(s)}{U_{\varphi}(s)} \Bigg) \Bigg).
\end{equation}
As we have seen in the proof of the Theorem \ref{TheoremA}, the explicit forms of $U(s)$ and $V(s)$ in $\omega_{{\rm{CY}}}$ do not play any role in the arguments. In fact, the proof follows only using the positivity of $U$ and $V$, the monotonicity of the phase function in the radial variable, and the linear growth of $V\tan(H)$ at infinity. Since these properties continue to hold for $({\rm{Tot}}({\bf{E}}),\omega_{\varphi})$, the same global existence result holds in this setting. Once we have the explicit profile $\varphi$, the study of solutions to the dHYM equation on the Kähler manifold $({\rm{Tot}}({\bf{E}}),\omega_{\varphi})$ can be approached by ODE methods similarly to the Calabi-Yau case of $({\rm{Tot}}({\bf{K}}_{X_{P}}),\omega_{{\rm{CY}}})$.

\subsection*{Examples}

In what follows, we present detailed examples which illustrate how our main results can be applied concretely.

\subsubsection*{dHYM Connections on the Eguchi-Hanson Space}

Consider $X_{B} = {\rm SL}_{2}(\mathbbm{C})/B $ equipped with the unique ${\rm SU}(2)$-invariant Kähler-Einstein metric $\omega_{0}$, satisfying ${\rm Ric}(\omega_{0}) = \omega_{0}$. Following Remark \ref{canonicalflag}, we take
\begin{equation}
\omega_{0} = 2\pi\langle \delta_{B}, \alpha^{\vee}\rangle{\bf{\Omega}}_{\alpha} = 4\pi{\bf{\Omega}}_{\alpha},
\end{equation}
where $\alpha$ denotes the unique positive root of $\mathfrak{sl}_{2}(\mathbbm{C})$ and ${\bf{\Omega}}_{\alpha}$ is the ${\rm{SU}}(2)$-invariant $(1,1)$-form generating $\mathrm{Pic}(X_{B})$, normalized by $\int_{\mathbbm{P}^{1}} {\bf{\Omega}}_{\alpha} = \langle \varpi_{\alpha}, \alpha^{\vee}\rangle = 1$. In particular, $[\omega_{0}] = 4\pi c_{1}(\mathscr{O}_{\alpha}(1))$, and 
\begin{equation}
\displaystyle c_{1}(\mathscr{O}_{\alpha}(k)) = \frac{k}{4\pi}[\omega_{0}].
\end{equation}
From above, since $\delta_{B} = \alpha = 2 \varpi_{\alpha}$, we have ${{\bf{K}}}_{X_{B}}^{-1} = \mathscr{O}_{\alpha}(2)$. Thus, from Remark \ref{proj}, we have
\begin{equation}
H^{0}(X_{B}, {\bf{K}}_{X_{B}}^{-1})^{\ast} \cong \mathfrak{sl}_{2}(\mathbbm{C}).
\end{equation}
From the projective embedding $X_{B} \hookrightarrow \mathbbm{P}(\mathfrak{sl}_{2}(\mathbbm{C}))$, we obtain the following projective algebraic realization 
\begin{equation}
X_{B} \cong \Bigg \{ \Bigg [ \begin{pmatrix}
x & y \\
z & -x 
\end{pmatrix} \Bigg ] \in \mathbbm{P}(\mathfrak{sl}_{2}(\mathbbm{C})) \ \Bigg | \ x^{2} + zy = 0 \Bigg \}.
\end{equation}
In particular, identifying $\mathbbm{C}^{3} \cong \mathfrak{sl}_{2}(\mathbbm{C})$, we have the following crepant resolution of isolated singularity
\begin{equation}
\mathscr{R} \colon  Z = {\rm{Tot}}({\bf{K}}_{X_{B}}) \to {\rm{Aff}}(X_{B}) = \Big \{ (x,y,z) \in \mathbbm{C}^{3}  \ \  \Big |  \ \ x^{2} + zy = 0 \Big \}.
\end{equation}
It is worth pointing out that $ {\rm{Aff}}(X_{B}) = \overline{\mathcal{O}_{{\rm{min}}}}$, where $\mathcal{O}_{{\rm{min}}} \subset \mathfrak{sl}_{2}(\mathbbm{C})$ is the minimal nilpotent orbit. Also, the manifold $Z$ equipped with the Calabi ansatz metric $\omega_{{\rm{CY}}}$ is known as the Eguchi-Hanson space.
\begin{center}
\begin{figure}[H]
\centering\includegraphics[scale = .27]{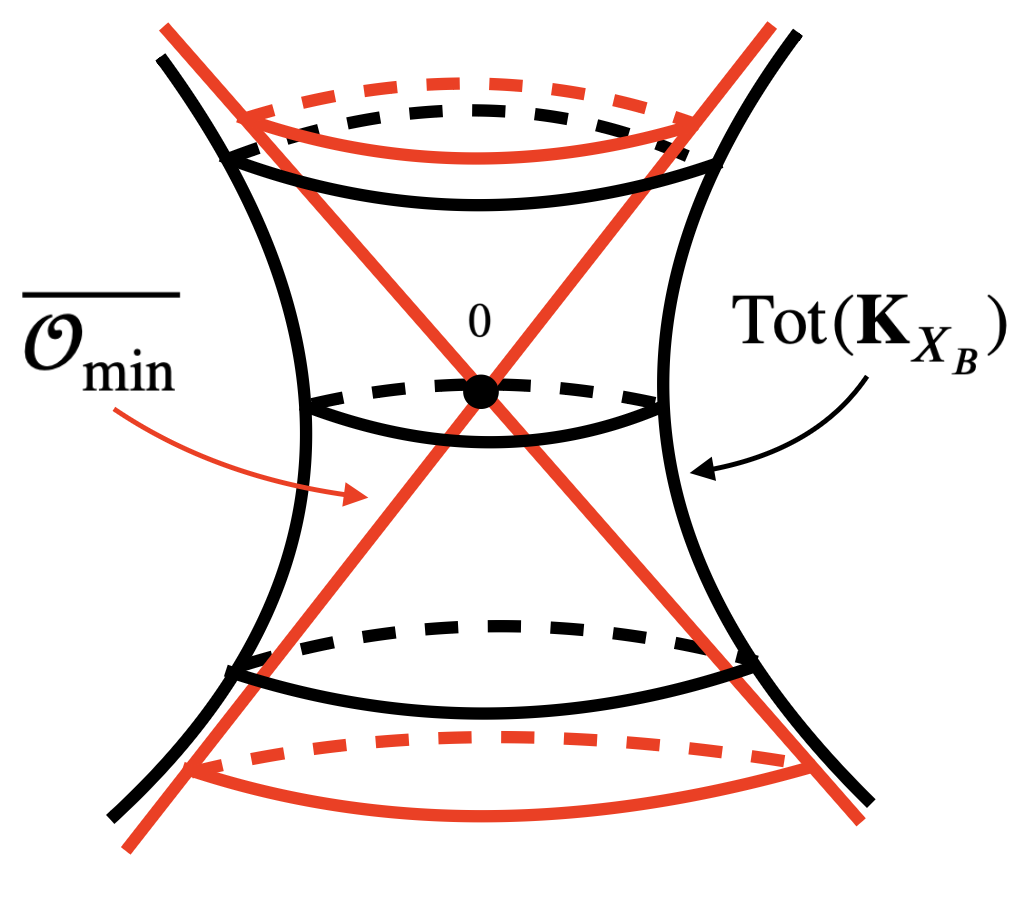}
\caption{Cartan-Remmert reduction as a crepant resolution of isolated singularity.}
\label{singularcone}
\end{figure}
\end{center}
The variety $X_{B}$ has complex dimension $n=1$, which fixes the base scale factor of the Calabi ansatz as $U(s) = \sqrt{1 + c s}$ and the fiber factor corresponding to the scalar ODE as $V(s) = U'(s)$. Let $\mathcal{L} \to Z$ be the pullback of the line bundle $\mathscr{O}_{\alpha}(k) \to X_{B}$. Since we have a single base eigenvalue dictated by the positive root, by Proposition \ref{eigenvalueatorigin} we obtain
\begin{equation}
\label{eq:q-value}
q := {\bf q}(\omega_{0}^{-1} \circ \chi) 
= \frac{\langle \lambda([\chi]), \alpha^{\vee}\rangle}{\langle \lambda([\omega_{0}]), \alpha^{\vee}\rangle} 
= \frac{k}{4\pi},
\end{equation}
where we used $\lambda([\chi]) = k\varpi_{\alpha}$, since $[\chi] = c_{1}(\mathscr{O}_{\alpha}(k))$ and $\lambda([\omega_{0}]) = 4\pi\varpi_{\alpha}$. Hence, following Theorem \ref{TheoremA} and Corollary \ref{CorollaryA}, we can find a Hermitian connection on $\mathcal{L}$ solving the dHYM equation for every phase
\begin{equation}
\hat{\Theta}_{{\rm tot}}  \in \big (\arctan(q) - \pi/2,\arctan(q) + \pi/2 \big).
\end{equation}
Adopting $k > 0$, we have $q > 0$, and we can choose $\hat{\Theta}_{{\rm tot}} = \pi/2$. Using the cotangent identity $\tan(\pi/2 - x) = \cot(x)$, we have
\begin{equation}
    \frac{d\psi}{dU} = \frac{d\psi/ds}{dU/ds} = \frac{V(s) \cot\big ( \arctan\big (\frac{q + \psi}{U}\big) \big)}{V(s)} = \frac{U}{q + \psi}.
\end{equation}
The integration of this separable differential equation, utilizing the regularity condition at the base $\psi(0) = 0$ to fix the integration constant at $U(0)=1$, produces an explicit solution 
\begin{equation}
    \psi(s) = -q + \sqrt{q^{2} - 1 + U(s)^{2}} = -q + \sqrt{q^{2} + cs}.
\end{equation}
As $s\to+\infty$, expanding the square root yields
\begin{equation}
\sqrt{q^{2}+cs}
= \sqrt{cs}\left(1+\frac{q^{2}}{cs}\right)^{1/2}
= \sqrt{cs}+\frac{q^{2}}{2\sqrt{cs}}+\mathcal{O}\!\left(s^{-3/2}\right).
\end{equation}
Therefore,
\begin{equation}
\psi(s)
= \sqrt{cs}-q+\frac{q^{2}}{2\sqrt{cs}}+\mathcal{O}\!\left(s^{-3/2}\right).
\end{equation}
In particular, the dominant term satisfies $\psi(s)\sim \sqrt{cs}\sim U(s)$, as $s\to+\infty$. Moreover, since
\begin{equation}
U(s)=\sqrt{1+cs}
= \sqrt{cs}+\frac{1}{2\sqrt{cs}}+\mathcal{O}\!\left(s^{-3/2}\right),
\end{equation}
we obtain the refined comparison
\begin{equation}
\psi(s)
= U(s)-q+\frac{q^{2}-1}{2\sqrt{cs}}+\mathcal{O}\!\left(s^{-3/2}\right),
\end{equation}
which describes the asymptotic behavior of the dHYM solution on ${{\bf{K}}}_{X_{B}}$.

\begin{remark}[A dHYM solution which is not HYM]
For the line bundle $\mathscr{O}_{\alpha}(k)$ with $k > 4\pi$ (so that $q > 1$) and the topological phase fixed at $\hat{\Theta}_{{\rm tot}} = \pi/2$, the radial momentum was found to be 
\begin{center}
$\psi(s) = -q + \sqrt{q^{2} - 1 + U(s)^{2}}$. 
\end{center}
Consequently, the trace of the endomorphism is given by
\begin{equation}
    \Lambda_{base}(s) + \Lambda_{{\rm fib}}(s) = \frac{\sqrt{q^{2} - 1 + U(s)^{2}}}{U(s)} + \frac{U(s)}{\sqrt{q^{2} - 1 + U(s)^{2}}}.
\end{equation}
This trace is strictly a non-constant function of the radial scale $U(s)$. Therefore, the connection which solves the nonlinear dHYM equation strictly fails to be HYM. In order to solve both geometric PDE's simultaneously, we can choose 
\begin{equation}
\psi(s) = q(U(s) - 1) = q(\sqrt{1+cs} - 1) \sim qU(s).
\end{equation}
As we see, the solution above grows $q$ times faster at infinity than the strictly dHYM solution.
\end{remark}

\begin{remark}
In the last example, we have $X_{B} \cong \mathbbm{P}^{1}$, that is, $Z = {\bf{K}}_{\mathbbm{P}^{1}}$. In this setting, solutions to the dHYM equation was also obtained in \cite{fowdar2024examples} using the cohomogeneity-one symmetry of ${\bf{K}}_{\mathbbm{P}^{1}}$. In view of this example, our main result provides a generalization of this approach which allows one to carry the construction of dHYM connections on every line bundle over ${\bf{K}}_{\mathbbm{P}^{n}} = \mathscr{O}(-n-1)$, for every $n > 0$ (cf. \cite{fowdar2024examples, fowdar2024explicit}). Here the flag variety $\mathbbm{P}^{n}$ can be represented by the following diagram 
\begin{equation*}
{\dynkin[labels={\alpha_{1},\alpha_{2},\alpha_{n}},scale=2.5]A{x*...*}} 
\end{equation*}
It means that $H^{2}(\mathbbm{P}^{n},\mathbbm{Z}) = \mathbbm{Z}[{\bf{\Omega}}_{\alpha_{1}}]$. As it can be seen in the proof of Theorem \ref{TheoremA}, the fundamental result which allows to establish this generalization is the complete description of Lagrangian phase of invariant real $(1,1)$-forms on flag varieties provided in \cite{correa2026deformed}.
\end{remark}

\subsubsection*{Examples with Picard Number Greater than One}

The description of the solutions to the dHYM on $K_{X}$ for flag varieties $X$ with Picard number one can be naturally described as in the case of $\mathbbm{P}^{n}$. This class of examples includes:
\begin{enumerate}
\item[$\bullet$] ${\rm{LGr}}_{4}(\mathbbm{R}^{8}) = \dynkin[scale=1.5]{C}{***x}$ \ \ (Lagrangian Grassmannian).
\item[$\bullet$] $\mathcal{Q}_{7} = \dynkin[scale=1.5]{B}{x***}$ \ \ (Complex Projective Quadric).
\item[$\bullet$] $\mathbb{S}_{5} = \dynkin[scale=1.5]{D}{****x}$ \ \ (Spinor Variety).
\item[$\bullet$] ${\rm{Gr}}_{2}(\mathbbm{C}^{4}) = \dynkin[scale=1.5]{D}{*x*}$  \ \ (Klein Quadric).
\end{enumerate}
Let us consider now an example with Picard number greater than one. Let $X_{B} = {\rm{SL}}_{3}(\mathbbm{C})/B$ the full flag variety associated to ${\rm{SL}}_{3}(\mathbbm{C})$. The diagram representing the variety in this case is
\begin{equation*}
{\dynkin[labels={\alpha_{1},\alpha_{2}},scale=2.5]A{xx}}, 
\end{equation*}
which means that $H^{2}(X_{B},\mathbbm{Z}) = \mathbbm{Z}[{\bf{\Omega}}_{\alpha_{1}}] \oplus \mathbbm{Z}[{\bf{\Omega}}_{\alpha_{2}}]$. Also, we have 
\begin{equation}
\delta_{B} = 2\alpha_{1} + 2\alpha_{2} = 2(\varpi_{\alpha_{1}} + \varpi_{\alpha_{2}}).
\end{equation}
Therefore, we conclude that 
\begin{equation}
{\bf{K}}_{X_{B}}^{-1} = \mathscr{O}_{\alpha_{1}}(2) \otimes \mathscr{O}_{\alpha_{2}}(2).
\end{equation}
From Borel-Weil theorem (see Remark \ref{proj}), we have $H^{0}(X_{B},{\bf{K}}_{X_{B}}^{-1})^{\ast} \cong V(\delta_{B})$ and the projective embedding $X_{B} \hookrightarrow \mathbb{P}(V(\delta_{B}))$. This embedding can be described as follows. At first, we notice that $\delta_{B} = 2\delta^{+}$, where 
\begin{center}
$\delta^{+} = \varpi_{\alpha_{1}} + \varpi_{\alpha_{2}} = \frac{1}{2}\sum_{\alpha \in \Phi^{+}}\alpha$. 
\end{center}
Now we notice that ${\bf{L}} = \mathscr{O}_{\alpha_{1}}(1) \otimes \mathscr{O}_{\alpha_{2}}(1)$ is very ample and 
\begin{equation} 
H^{0}(X_{B},{\bf{L}})^{\ast} \cong V(\delta^{+}) \subset \mathbbm{C}^{3} \otimes \textstyle{\bigwedge^{2}\mathbbm{C}^{3}}.
\end{equation}
Therefore, we have a projective embedding $\iota_{0} \colon X_{B} \hookrightarrow \mathbb{P}(V(\delta^{+}))$, $gB \mapsto [\varrho(g)v_{\delta^{+}}^{+}]$, where
\begin{equation}
v_{\delta^{+}}^{+} = {\bf{e}}_{1} \otimes ({\bf{e}}_{1} \wedge {\bf{e}}_{2}).
\end{equation}
In this case, it follows that 
\begin{center}
$X_{B} \cong  \Big \{ [z \otimes w] \in \mathbb{P}(V(\delta^{+})) \ \Big | \  z_{1}w_{23} - z_{2}w_{13} + z_{3}w_{12} = 0\Big\}$,
\end{center}
where $z = (z_{1},z_{2},z_{3})$ e $w = w_{12} {\bf{e}}_{1} \wedge {\bf{e}}_{2} + w_{13}{\bf{e}}_{1}\wedge {\bf{e}}_{3} + \wedge w_{23}{\bf{e}}_{2} \wedge {\bf{e}}_{3} = (w_{12},w_{13},w_{23})$. Considering the Veronese embedding
\begin{equation}
v_{2} \colon \mathbbm{P}(V(\delta^{+})) \hookrightarrow \mathbbm{P}({\rm{Sym}}^{2}V(\delta^{+})), \ \ [v] \mapsto [v \odot v],
\end{equation}
and observing that $V(\delta_{B}) = V(2\delta^{+}) \subset {\rm{Sym}}^{2}V(\delta^{+}) \subset V(\delta^{+}) \otimes V(\delta^{+})$, we have the desired projective embedding $v_{2} \circ \iota_{0} \colon X_{B} \to \mathbbm{P}(V(\delta_{B}))$ and the algebraic realization 
\begin{equation}
X_{B} \cong \Big \{ [(z \otimes w) \odot (z \otimes w)] \ \Big | \ z_{1}w_{23} - z_{2}w_{13} + z_{3}w_{12} = 0\Big \}.
\end{equation}
From the above ideas, we have the crepant resolution provided by the Cartan-Remmert reduction
\begin{equation}
\mathscr{R} \colon {\rm{Tot}}({\bf{K}}_{X_{B}}) \to {\rm{Aff}}(X_{B}) = \Big \{ (z \otimes w) \odot (z \otimes w) \ \Big | \ z_{1}w_{23} - z_{2}w_{13} + z_{3}w_{12} = 0\Big \}.
\end{equation}
In particular, we notice that 
\begin{equation}
\omega_{0} = 4\pi \big ( {\bf{\Omega}}_{\alpha_{1}} + {\bf{\Omega}}_{\alpha_{2}}\big),
\end{equation}
satisfies ${\rm{Ric}}(\omega_{0}) = \omega_{0}$. From these data, we can consider the Calabi-Yau metric on ${\rm{Tot}}({\bf{K}}_{X_{B}})$ obtained from the Calabi ansatz, i.e.,
\begin{center}
$\displaystyle \omega_{{\rm{CY}}}= U(s) p^{\ast}\omega_{0} + V(s) \frac{\sqrt{-1} \partial s \wedge \bar{\partial} s}{s},$ 
\end{center}
such that $U(s) = (1 + c s)^{\frac{1}{4}}$ and $V(s) = U'(s)$. Given $\mathcal{L} = p^{\ast}{\bf{L}} \to {\bf{K}}_{X_{B}}$, such that  
\begin{equation}
{\bf{L}} = \mathscr{O}_{\alpha_{1}}(a) \otimes \mathscr{O}_{\alpha_{2}}(b),
\end{equation}
we have the unique ${\rm{SU}}(3)$-invariant representative $\chi = a {\bf{\Omega}}_{\alpha_{1}} + b {\bf{\Omega}}_{\alpha_{2}} \in c_{1}({\bf{L}})$ satisfying
\begin{equation}
{\bf{q}}_{1} = {\bf{q}}_{\alpha_{1}}(\omega^{-1} \circ \chi) = \frac{a}{4\pi}, \ \ \ {\bf{q}}_{2} = {\bf{q}}_{\alpha_{2}}(\omega^{-1} \circ \chi) = \frac{b}{4\pi}, \ \ \ {\bf{q}}_{3} = {\bf{q}}_{\alpha_{3}}(\omega^{-1} \circ \chi) = \frac{a+b}{8\pi},
\end{equation}
where $\Phi^{+} = \{\alpha_{1},\alpha_{2}, \alpha_{3} = \alpha_{1} + \alpha_{2}\}$ is the positive root system of $\mathfrak{sl}_{3}(\mathbbm{C})$. From this, the dHYM equation 
\begin{equation}
{\rm{Im}}\big (\omega_{{\rm{CY}}} + \sqrt{-1}\Upsilon_{\phi} \big )^{4} = \tan(\hat{\Theta}_{{\rm{tot}}}) {\rm{Re}}\big (\omega_{{\rm{CY}}} + \sqrt{-1}\Upsilon_{\phi} \big )^{4}
\end{equation}
where $\Upsilon_{\phi} = p^{\ast}\chi - \sqrt{-1}\partial \bar{\partial} \phi(s) \in c_{1}(\mathcal{L})$ reduces to the following IVP
\begin{equation}
\frac{d}{ds}\psi(s) = V(s) \tan\Bigg ( \displaystyle \hat{\Theta}_{{\rm{tot}}} - \sum_{j = 1}^{3} \arctan\Bigg( \frac{{\bf{q}}_{j} + \psi(s)}{U(s)} \Bigg) \Bigg), \ \ \ \ \psi(0) = 0.
\end{equation}
From Theorem \ref{TheoremA}, if we take $\hat{\Theta}_{{\rm{tot}}} \in \big (\Theta_{\omega_{0}}(\chi) -\tfrac{\pi}{2}, \Theta_{\omega_{0}}(\chi) + \tfrac{\pi}{2}\big)$, the IVP above admits a globally defined (unique) smooth solution. Let us examine some particular cases where we can describe the solution explicitly. 

\begin{itemize}
\item[(A)] Consider the case that $a = b$. In this case, we have ${\bf{q}}_{j} = a/4\pi = q$, $j = 1,2,3$, and the associated ODE becomes 
\begin{equation}
\frac{d}{ds}\psi(s) = V(s) \tan\Bigg ( \displaystyle \hat{\Theta}_{{\rm{tot}}} - 3 \arctan\Bigg( \frac{q + \psi(s)}{U(s)} \Bigg) \Bigg), \ \ \ \ \psi(0) = 0.
\end{equation}
Observing that $\Theta_{\omega_{0}}(\chi) = 3 \arctan(q)$, we have 
\begin{equation}
\hat{\Theta}_{{\rm{tot}}} = \frac{\pi}{2} \in \big (\Theta_{\omega_{0}}(\chi) -\tfrac{\pi}{2}, \Theta_{\omega_{0}}(\chi) + \tfrac{\pi}{2}\big) \iff 0 < q < \sqrt{3},
\end{equation}
i.e., we can take $\hat{\Theta}_{{\rm{tot}}} = \frac{\pi}{2}$ if, and only if, $a = 1,2,\ldots,21$. Therefore, considering $0 < q < \sqrt{3}$, and choosing the phase $\hat{\Theta}_{{\rm{tot}}} = \frac{\pi}{2}$, we have 
\begin{equation}
\frac{d\psi}{dU} = \frac{d\psi/ds}{dU/ds} = \cot(3\arctan(y)), \ \ y = \frac{q + \psi}{U}.
\end{equation}
From above, since
\begin{equation}
\cot(3\arctan(y)) = \frac{1-3y^{2}}{y(3-y^{2})} \ \ \ \text{and} \ \ \ \frac{d\psi}{dU} = U\frac{dy}{dU} + y,
\end{equation}
we obtain 
\begin{equation}
U\frac{dy}{dU} = \frac{1-3y^{2}}{y(3-y^{2})} -y = \frac{y^{4}-6y^{2}+1}{y(3-y^{2})}.
\end{equation}
Observing that $\frac{dy}{ds} = \frac{dy}{dU} \frac{dU}{ds}$, we conclude that 
\begin{equation}
U \frac{dy}{ds} = \frac{dU}{ds} \Big (\frac{y^{4}-6y^{2}+1}{y(3-y^{2})}\Big ) \iff \frac{dU}{ds}\frac{1}{U} = \frac{y(3-y^{2})}{y^{4}-6y^{2}+1}\frac{dy}{ds}.
\end{equation}
Integrating both sides of the last expression in the right-hand side yields 
\begin{equation}
\log(U) = -\frac{1}{4} \log \big ( y^{4} - 6y^{2} + 1\big ) + C.
\end{equation}
From the initial conditions $U(0)=1$ and $y(0) = q$, we have
\begin{equation}
U^{4}( y^{4} - 6y^{2} + 1) = q^{4} - 6q^{2} + 1.
\end{equation}
Thus, we obtain
\begin{equation}
y(s)^{2} = 3 \pm \sqrt{8 +\frac{q^{4} - 6q^{2} + 1}{U(s)^{4}}}.
\end{equation}
Since $y(0)^{2} = q^{2}$, $q^{4} - 6q^{2} + 9 = (q^{2}-3)^{2}$, and $q^{2} < 3$, it follows that 
\begin{equation}
y(s) = \pm \sqrt{ 3 - \sqrt{8 +\frac{q^{4} - 6q^{2} + 1}{U(s)^{4}}}}.
\end{equation}
Therefore, from the initial condition $\psi(0) = 0$ and the relation $\psi(s) = U(s)y(s) - q$, we obtain the exact solution
\begin{equation}
\psi(s) = - q + U(s)\sqrt{ 3 - \sqrt{8 +\frac{q^{4} - 6q^{2} + 1}{U(s)^{4}}}}.
\end{equation}
From the above expression we see that $\psi(s) \sim (\sqrt{2}-1)U(s) - q$, as $s \to +\infty$. We notice that the dHYM solution $\Upsilon_{\phi}$, such that $\phi(s) = -\int_{0}^{s}\frac{\psi(u)}{u}du$, satisfies
\begin{equation}
\Lambda_{\omega_{{\rm{CY}}}}(\Upsilon_{\phi}) = 3y(s) + \frac{\psi'(s)}{V(s)}.
\end{equation}
Since the function on the right-hand side above is not constant, we conclude that the connection $\nabla$ on $\mathcal{L}$, satisfying $\sqrt{-1}F_{\nabla} = \Upsilon_{\phi}$ is dHYM but not HYM.

\item[(B)] Considering $a = b$ as in the previous example, the HYM equation $\Lambda_{\omega_{{\rm{CY}}}}(\Upsilon_{\phi}) = C$ yields
\begin{equation}
C = \Lambda_{\omega_{{\rm{CY}}}}(\Upsilon_{\phi}) = 3\frac{q + \psi}{U} + \frac{\psi'}{V},
\end{equation}
recall that $V = U'$. This is a first-order linear ODE and admits a family of smooth global solutions
\begin{equation}
\psi_{C}(s) = \frac{C}{4}U(s) - q + \Big (q - \frac{C}{4}\Big ) \frac{1}{U(s)^{3}}. 
\end{equation}
From above, we obtain a Hermitian connection $\nabla$ on $\mathcal{L}$ such that $\sqrt{-1}F_{\nabla} =\Upsilon_{\phi_{C}}$, with $\phi_{C}(s) = -\int_{0}^{s}\frac{\psi_{C}(u)}{u} du$, solving the HYM equation $\Lambda_{\omega_{{\rm{CY}}}}(\sqrt{-1}F_{\nabla}) = C$. In this case, considering
\begin{equation}
x_{C}(s) := \frac{q + \psi_{C}(s)}{U(s)} \ \ \ \ \text{and} \ \ \ \  y_{C}(s) := \frac{\psi_{C}'(s)}{V(s)},
\end{equation}
the dHYM equation for $\Upsilon_{\phi_{C}}$ reads
\begin{equation}
\Theta(s) := 3 \arctan(x_{C}(s)) + \arctan(y_{C}(s)) = \hat{\Theta}_{{\rm{tot}}}.
\end{equation}
If $C = 4q$, we have $\psi_{C}(s) = q(U(s) - 1)$, which implies that $\Upsilon_{\phi_{C}}$ is also a solution to the dHYM equation with phase $\hat{\Theta}_{{\rm{tot}}} = 4 \arctan(q)$, as predicted by Corollary \ref{CorollaryA}. On the other hand, if $C \neq 4q$, the function $\Theta(s)$ is not constant, i.e., the HYM connection $\nabla$ fails to solve the dHYM equation.
\end{itemize}
To the best of our knowledge, Example (B) above provides the first explicit example of an HYM connection which is not dHYM, on a holomorphic line bundle over a non-compact, non-toric Calabi-Yau manifold. Beyond its role as a counterexample, this computation illustrates the effectiveness of the ODE reduction developed in Theorem \ref{TheoremA}.

\subsection*{Acknowledgments.} E. M. Correa is supported by S\~{a}o Paulo Research Foundation FAPESP grant 2025/18843-1. 

\subsection*{Generative AI Disclosure} The author used the online version of DeepSeek as an auxiliary tool to perform some computations, search for references, and make minor edits during the development of the paper. The author takes full responsibility for all mathematical arguments, and the manuscript was entirely written and revised by the author.

\subsection*{Data Availability} Data sharing not applicable to this article as no datasets were generated or analysed during the current study.

\subsection*{Conflict of Interest Statement} The author declares that there is no conflict of interest.

\bibliographystyle{alpha}
\bibliography{bibli}

\begin{thebibliography}{MMMS00}

\bibitem[AB03]{AZAD}
Hassan Azad and Indranil Biswas.
\newblock Quasi-potentials and {K}\"{a}hler-{E}instein metrics on flag manifolds. {II}.
\newblock {\em J. Algebra}, 269(2):480--491, 2003.

\bibitem[Akh95]{Akhiezer}
Dmitri~N. Akhiezer.
\newblock {\em Lie group actions in complex analysis}.
\newblock Aspects of Mathematics, E27. Friedr. Vieweg \& Sohn, Braunschweig, 1995.

\bibitem[Bal23]{ballal2023supercritical}
Aashirwad Ballal.
\newblock The supercritical deformed {Hermitian} {Yang-Mills} equation on compact projective manifolds.
\newblock {\em Illinois Journal of Mathematics}, 67(1):73--99, 2023.

\bibitem[BR62]{BorelRemmert}
A.~Borel and R.~Remmert.
\newblock \"{U}ber kompakte homogene {K}\"{a}hlersche {M}annigfaltigkeiten.
\newblock {\em Math. Ann.}, 145:429--439, 1961/62.

\bibitem[Cal79]{Calabi1979}
Eugenio Calabi.
\newblock M\'etriques k\"ahl\'eriennes et fibr\'es holomorphes.
\newblock {\em Annales scientifiques de l'\'Ecole Normale Sup\'erieure}, 12(2):269--294, 1979.

\bibitem[Che21]{Chen2021j}
Gao Chen.
\newblock The j-equation and the supercritical deformed {Hermitian--Yang--Mills} equation.
\newblock {\em Inventiones mathematicae}, 225:529--602, 2021.

\bibitem[CJY20]{Collins2020}
Tristan~C Collins, Adam Jacob, and Shing-Tung Yau.
\newblock $(1, 1) $ forms with specified {Lagrangian} phase: a priori estimates and algebraic obstructions.
\newblock {\em Cambridge Journal of Mathematics}, 8(2):407--452, 2020.

\bibitem[CLT24]{chu2024nakai}
Jianchun Chu, Man-Chun Lee, and Ryosuke Takahashi.
\newblock A {Nakai--Moishezon} type criterion for supercritical deformed {Hermitian--Yang--Mills} equation.
\newblock {\em Journal of Differential Geometry}, 126(2):583--632, 2024.

\bibitem[Cor19]{Correa}
Eder~M. Correa.
\newblock Homogeneous contact manifolds and resolutions of {C}alabi-{Y}au cones.
\newblock {\em Comm. Math. Phys.}, 367(3):1095--1151, 2019.

\bibitem[Cor24]{correa2024dhym}
Eder~M Correa.
\newblock {DHYM} connections on higher rank holomorphic vector bundles over p (tp 2).
\newblock {\em Mathematische Zeitschrift}, 308(2):36, 2024.

\bibitem[Cor26]{correa2026deformed}
Eder~M Correa.
\newblock Deformed hermitian yang--mills equation on rational homogeneous varieties.
\newblock {\em Documenta Mathematica}, 2026.

\bibitem[COSD26]{charbonneau2026deformed}
Benoit Charbonneau, Gon{\c{c}}alo Oliveira, and Rosa Sena-Dias.
\newblock Deformed hermitian-yang-mills equation on the manifold of full flags.
\newblock {\em arXiv preprint arXiv:2607.08622}, 2026.

\bibitem[CXY18]{Collins2018deformed}
Tristan Collins, Dan Xie, and Shing-Tung Yau.
\newblock The deformed {Hermitian--Yang--Mills} equation in geometry and physics.
\newblock {\em Geometry and physics}, 1:69--90, 2018.

\bibitem[Fow24a]{fowdar2024examples}
Udhav Fowdar.
\newblock Examples of deformed spin (7)-instantons/donaldson--thomas connections.
\newblock {\em Communications in Mathematical Physics}, 405(8):183, 2024.

\bibitem[Fow24b]{fowdar2024explicit}
Udhav Fowdar.
\newblock Explicit abelian instantons on s1-invariant k{\"a}hler einstein 6-manifolds.
\newblock {\em Journal of Geometry and Physics}, 203:105269, 2024.

\bibitem[FW01]{FultonWoodward}
William Fulton and Chris Woodward.
\newblock On the quantum product of {Schubert classes}.
\newblock {\em Journal of Algebraic Geometry}, 13:641--661, 2001.

\bibitem[Gra62]{grauert1962modifikationen}
Hans Grauert.
\newblock {\"U}ber modifikationen und exzeptionelle analytische mengen.
\newblock {\em Mathematische Annalen}, 146(4):331--368, 1962.

\bibitem[HS02]{HwangSinger2002}
Andrew~D Hwang and Michael~A Singer.
\newblock A momentum construction for circle-invariant k{\"a}hler metrics.
\newblock {\em Transactions of the American Mathematical Society}, 354(6):2285--2325, 2002.

\bibitem[Hum72]{Humphreys}
James~E. Humphreys.
\newblock {\em Introduction to {L}ie algebras and representation theory}.
\newblock Graduate Texts in Mathematics, Vol. 9. Springer-Verlag, New York-Berlin, 1972.

\bibitem[Hum75]{HumphreysLAG}
James~E. Humphreys.
\newblock {\em Linear algebraic groups}.
\newblock Graduate Texts in Mathematics, No. 21. Springer-Verlag, New York-Heidelberg, 1975.

\bibitem[Jac22]{jacob2022deformed}
Adam Jacob.
\newblock The deformed hermitian-yang-mills equation and level sets of harmonic polynomials.
\newblock {\em arXiv preprint arXiv:2204.01875}, 2022.

\bibitem[Jan03]{jantzen2003representations}
Jens~Carsten Jantzen.
\newblock {\em Representations of algebraic groups}, volume 107.
\newblock American Mathematical Soc., 2003.

\bibitem[JY17]{JacobYau2017}
Adam Jacob and Shing-Tung Yau.
\newblock A special {Lagrangian} type equation for holomorphic line bundles.
\newblock {\em Mathematische Annalen}, 369:869--898, 2017.

\bibitem[Kob87]{Kobayashi+1987}
Shoshichi Kobayashi.
\newblock {\em Differential Geometry of Complex Vector Bundles}.
\newblock Princeton University Press, Princeton, 1987.

\bibitem[LB18]{Flagvarieties}
V.~Lakshmibai and Justin Brown.
\newblock {\em Flag varieties}, volume~53 of {\em Texts and Readings in Mathematics}.
\newblock Hindustan Book Agency, Delhi, 2018.
\newblock An interplay of geometry, combinatorics, and representation theory, Second edition of [ MR2474907].

\bibitem[LR08]{Lakshmibai2}
Venkatramani Lakshmibai and Komaranapuram~N. Raghavan.
\newblock {\em Standard monomial theory}, volume 137 of {\em Encyclopaedia of Mathematical Sciences}.
\newblock Springer-Verlag, Berlin, 2008.
\newblock Invariant theoretic approach, Invariant Theory and Algebraic Transformation Groups, 8.

\bibitem[LYZ00]{leung2000special}
Naichung~Conan Leung, Shing-Tung Yau, and Eric Zaslow.
\newblock From special {Lagrangian to Hermitian-Yang-Mills} via {Fourier-Mukai} transform.
\newblock {\em arXiv preprint math/0005118}, 2000.

\bibitem[Mat72]{MATSUSHIMA}
Yozo Matsushima.
\newblock Remarks on {K}\"{a}hler-{E}instein manifolds.
\newblock {\em Nagoya Math. J.}, 46:161--173, 1972.

\bibitem[MJ05]{malley2005theory}
Robert~E Malley~Jr.
\newblock The theory of differential equations: Classical and qualitative.: Classical and qualitative.
\newblock {\em SIAM Review}, 47(1):185, 2005.

\bibitem[MMMS00]{marino2000nonlinear}
Marcos Marino, Ruben Minasian, Gregory Moore, and Andrew Strominger.
\newblock Nonlinear instantons from supersymmetric p-branes.
\newblock {\em Journal of High Energy Physics}, 2000(01):005, 2000.

\bibitem[Pin19]{Pingali2019}
Vamsi~P Pingali.
\newblock A note on the deformed {Hermitian Yang-Mills PDE}.
\newblock {\em Complex Variables and Elliptic Equations}, 64(3):503--518, 2019.

\bibitem[SBBC02]{sale2002several}
On~Sale, MAA~Press Books, My~Bookshelf, and Successfully Copied.
\newblock Several complex variables with connections to algebraic geometry and lie groups.
\newblock 2002.

\bibitem[Ser54]{serre1954representations}
Jean-Pierre Serre.
\newblock Repr{\'e}sentations lin{\'e}aires et espaces homogenes k{\"a}hl{\'e}riens des groupes de lie compacts (d’apres armand borel et andr{\'e} weil).
\newblock {\em S{\'e}minaire Bourbaki}, 2(100):447, 1954.

\bibitem[She21]{sheu2021deformed}
Norman~Victor Sheu.
\newblock {\em The Deformed Hermitian-Yang-Mills Equation with Calabi Ansatz}.
\newblock PhD thesis, University of California, Davis, 2021.

\bibitem[Tak78]{MR528871}
Masaru Takeuchi.
\newblock Homogeneous {K}\"{a}hler submanifolds in complex projective spaces.
\newblock {\em Japan. J. Math. (N.S.)}, 4(1):171--219, 1978.

\bibitem[Tak20]{takahashi2020tan}
Ryosuke Takahashi.
\newblock Tan-concavity property for {Lagrangian} phase operators and applications to the tangent {Lagrangian} phase flow.
\newblock {\em International Journal of Mathematics}, 31(14):2050116, 2020.

\bibitem[VC10]{van2010ricci}
Craig Van~Coevering.
\newblock Ricci-flat k{\"a}hler metrics on crepant resolutions of k{\"a}hler cones.
\newblock {\em Mathematische Annalen}, 347(3):581--611, 2010.

\end{thebibliography}
\end{document}